\documentclass[15pt]{amsart}
\usepackage[all,cmtip]{xy}
\usepackage{amsthm}
\usepackage{xypic}
\usepackage{graphicx}
\usepackage{amscd}

\usepackage[colorlinks=false, linkcolor=red]{hyperref}

\newcommand{\M}{\overline{\mathcal{M}}}

\newcommand{\C}{\mathbb{C}}
\newcommand{\U}{\mathcal{U}}
\newcommand{\F}{\mathcal{F}}

\newcommand{\PP}{\mathbb{P}}

\newcommand{\OO}{\mathcal{O}}

\newtheorem{theorem}{Theorem}[section]
\newtheorem{thm}[theorem]{Theorem}
\newtheorem{cor}[theorem]{Corollary}
\newtheorem{lemm}[theorem]{Lemma}
\newtheorem{prop}[theorem]{Proposition}

\newtheorem{defi}[theorem]{Definition}

\newtheorem{condition}[theorem]{Conditions}
\newtheorem{remark}[theorem]{Remark}

\DeclareMathOperator{\Span}{Span}

\DeclareMathOperator{\Pic}{Pic}

\DeclareMathOperator{\Hom}{Hom}
\DeclareMathOperator{\Def}{Def}
\DeclareMathOperator{\Spec}{Spec}
\DeclareMathOperator{\Bl}{Bl}
\DeclareMathOperator{\Cone}{Cone}

\DeclareMathOperator{\Gal}{Gal}
\DeclareMathOperator{\Obs}{Obs}
\numberwithin{equation}{section}

\title []{Geometry of rational curves on low degree complete intersections}

\author {xuanyu pan}
\address{Department of Mathematics, Washington University in St.Louis, St.Louis, MO 63130}
\email{pan@math.wustl.edu}
\date{\today}
\begin{document}

\maketitle

\begin{abstract} For a smooth complete intersection $X$, we consider a general fiber $\F$ of the following evaluation map $ev$ of the Kontsevich moduli space
\[ev:\M_{0,m}(X,m)\rightarrow X^m\]
and the forgetful map $F:\F \rightarrow \M_{0,m}$. We prove that a general fiber of the map $F$ is a smooth complete intersection if $X$ is of low degree. The result sheds some light on the arithmetic and geometry of Fano complete intersections.
\end{abstract}
\tableofcontents

\section{Introduction}\label{notation}
Rationally connected varieties has been a central object of study in both algebraic geometry and arithmetic geometry. This notion can be seen as the algebraic analogue of path connectedness in topology. Rational simple connectedness is a notion proposed by B.~Mazur and systematically developed by de Jong-Starr, which can also be thought of as the analogue of simple connectedness in topology. This very sophisticated notion has been sucessfully used by de Jong-He-Starr to finish a proof of Serre's conjecture II (Pub. I.H.E.S) \cite{DS2} and a new proof of de Jong's period-index theorem \cite{Starr}.

One thing that we have learned from the paper of de Jong-He-Starr (Pub. I.H.E.S) \cite{DS2} is that a good understanding of the geometry of the moduli space of rational curves on a rationally connected variety is essential for applications in arithmetic. To add one more example of such application, we mention a beautiful theorem of B.~Hassett using "rational simple connectedness" to prove weak approximation. For more arithmetic applications, we refer to \cite{Brendan}, \cite{Tian1} and \cite{CZ}. On the other hand, a good understanding of these moduli spaces could also have applications in algebraic cycles on Fano manifolds, see \cite{TianZong} and \cite{PAN2}.

Given the importance of understanding the geometry of the moduli spaces of rational curves, it is somewhat unfortunate that we do not know much beyond the cases of homogeneous spaces (e.g. \cite{Zhuyi} \cite{KP}), or spaces with a large group action. For example, not much is known for complete intersections in the projective space. In this paper, we explore the moduli spaces of rational curves on complete intersections and fill the gap when the complete intersection are of low degree. 

In general, these moduli spaces are relatively easy to define as a functor, but very difficult to understand concretely. We combine classical projective geometry with advanced techniques of deformations of rational curves to give some concrete descriptions of these moduli spaces. In order to give these descriptions (see Theorem \ref{mainthm}), we need some notions. Let X be a projective variety in $\PP^n_{\mathbb{C}}$ with an effective curve class $\beta\in H_2(X,\mathbb{Z})$. The Kontsevich moduli space $\M_{0,m}(X,\beta)$ parameterizes data $(C, f, x_1,\dots,x_m)$ of
\begin{enumerate}
\item[(i)]
a proper, connected, at-worst-nodal, arithmetic genus $0$ curve $C$,
\item[(ii)]
an ordered collection $x_1,\dots,x_m$ of distinct smooth points of
$C$,
\item[(iii)]
and a morphism $f:C\rightarrow X$ with $f_*([C])=\beta$
\end{enumerate}
such that $(C,f,x_1,\dots,x_m)$ has only finitely many automorphisms. There is an evaluation morphism

 $$
\text{ev}_{\beta}:\M_{0,m}(X,\beta) \rightarrow X^m, \ \ (C,f,x_1,\dots,x_m)
\mapsto (f(x_1),\dots,f(x_m)).
$$
The space $\M_{0,m}(X,\beta)$ is a Deligne-Mumford stack, we refer to \cite{BM} for the construction.

\begin{defi} \label{hompoly}
Let $s$ and $d$ be natural numbers. We say that a family of homogeneous polynomials on $\PP^n$ is of type $T_1(d,s)$ if these polynomials consist of one polynomial of degree $d$, $s$ polynomials of degree $2$, $s$ polynomials of degree $3$, $\ldots$, and $s$ polynomials of degree $d-1$. We also denote the family of homogeneous polynomials by
\[
T_1(d,s)=\left(
  \begin{array}{ccccc}
    2 & 3 & \cdots & d-1 &  \\
    \vdots &\vdots &  & \vdots & d \\
    2 & 3 & \cdots & d-1 &  \\
  \end{array}
\right).
\]
We call a family of homogeneous polynomials on $\PP^n$ is of type $T_2(d,s)$ if the polynomials consist of one polynomial of degree $d$, $s$ polynomials of degree $1$, $s$ polynomials of degree $2$, $\ldots$ ,and $s$ polynomials of degree $d-1$. Denote the family of polynomials by
\[
T_2(d,s)=\left(
  \begin{array}{ccccc}
    1 & 2 & \cdots & d-1 &  \\
    \vdots &\vdots &  & \vdots & d \\
    1 & 2 & \cdots & d-1 &  \\
  \end{array}
\right).
\]
\end{defi}
In the following,  we say that a complete intersection of codimension $k$ is of type $(c_1,c_2,\ldots,c_k)$ if it is defined by $k$ homogeneous polynomials $F_j$ of degree $c_j$ for $j=1,\ldots,k$. It is convenient to use $T_i(d,s)$ for expressing the tuple $(c_1,c_2,\ldots,c_k)$. For instance, we denote by $(T_1(d_1,s),T_2(d_2,s))$ the tuple
\[ \left(
  \begin{array}{cccccccccc}
    2 & 3 & \cdots & d_1-1 &  & 1 &2 &\ldots &d_2-1\\
    \vdots &\vdots &  & \vdots & d_1 &\vdots &\vdots & &\vdots &d_2 \\
    2 & 3 & \cdots & d_1-1 &  &1 &2 &\ldots &d_2-1\\
  \end{array}
\right).\]
\begin{condition}\label{ineqcond} \label{forget} \label{phi}
Throughout this paper, we always assume that
\begin{itemize}
\item $X$ is a smooth complete intersection in $\PP^n_{\C}$ of type $(d_1,\ldots,d_c)$,
with $3\leq m\leq n+1$, $c\leq n$, $d_i\geq 2$ and $(c,d_1,\ldots,d_c)\neq (1,2)$ such that \[n+m(c-\sum\limits_{i=1}^c d_i)-c\geq 2;\footnote{This condition is to guarantee that the locus of $\F_t$ parametrizing the stable maps of maximal degeneration type is of dimension at least 2, cf. Corollary \ref{corintersectype}}\]
\item and $p_1 \ldots p_m$ are $m$ general points of $X$ and $\F$ is the general fiber of $ev_{m\alpha}$ over the general point $(p_1,\ldots,p_m)\in X^m$ where $\alpha$ is the homology class of a line;
\item and $\Phi: \F \dashrightarrow \PP^{n-m}=\PP^n/\Span(p_1,\ldots,p_m)$ \footnote{The map $\Phi$ is introduced in the proof of \cite[Lemma 6.4]{DS}. }   is a map associating to a point $(f,C,x_1,\ldots,x_m)\in \F $ the point $\Span(f(C)) \in \PP^{n-m}$\footnote{We denote by $\PP(V)/\PP(W)$ the projective space $\PP(V/W)$ for a flag $(W\subseteq V)$ of a vector space $V$  and denote by $\Span(D)$ the smallest projective subspaces containing the algebraic set $D(\subseteq \PP^n)$.}.

\end{itemize}
\end{condition}

Now, we are able to state our main theorem of this paper.
\begin{thm} \label{mainthm}
Assume the conditions as above. Let $\F_t$\footnote{The fiber $\F_t$ is possibly empty. However, in Section \ref{sectioncyclerelation}, we will see the conditions (\ref{ineqcond}) and Lemma \ref{lemmsection} guarantee that $\F_t$ is a non-empty, cf. Proposition \ref{cyclerelation}.} be the fiber of the forgetful map $F:\F \rightarrow \M_{0,m}$ (see \cite{FP}) over a general point $t$ in $\M_{0,m}$. 

Then, the map $\Phi$ is well defined on $\F_t$. Denote the line bundle $(\Phi|_{\F_t})^*\OO_{\PP^{n-m}}(1)$ by $\lambda|_{\F_t}$. The general fiber $\F_t$ is a smooth complete intersection in $\PP^{n-m(c-1)}$ of type\[\left(T_1(d_1,m), T_1(d_2,m),\ldots, T_{1}(d_{c-1},m) ,T_1(d_c,m) \right)\] via the complete linear system $|\lambda|_{\F_t}|:\F_t \hookrightarrow \PP^{n-m(c-1)}.$
\end{thm}

The importance of the geometry of the fibers $\F$ and $\F_t$ is discovered in the paper \cite{DS}. Roughly speaking, the rational connectedness of $\F$ and the existence of very twisting surface imply the rational simple connectedness of $X$. In the paper \cite{PAN1}, the author shows a conic's version of Theorem \ref{mainthm} (namely, $m=2$), cf. \cite[Theorem 1.1]{PAN1}. The result for conics strengthens some key results in \cite{DS} which is used to prove the (strongly) rational simple connectedness of low degree complete intersections, see \cite[Theorem 1.2]{DS} for details. The result for conics also leads a proof of strong approximation\footnote{Strong approximation is considered to be more difficult to show than weak approximation.} for low degree affine complete intersections over function fields by Q.~Chen and Y.~Zhu, see \cite{CZ}. Therefore, we expect that there will be more applications of our main theorem in arithmetic (eg. weak and strong approximation) and proving the (strongly) rational simple connectedness of complete intersections.

On the other hand, our main theorem instantly has some geometric applications. For example, we show the rational connectedness of $\F$, see Proposition \ref{propRC}, which is relating to rational simple connectedness of $X$; we reprove the formula in enumerative geometry for counting the number of twist cubics passing through $3$ general points on complete intersections due to A.~Beauville, see Proposition \ref{countcubic}, and we also find a new formula for linking conics, see Remark \ref{rmklinkconic}. Besides, Theorem \ref{mainthm} implies that the Picard group of $\F$ is finitely generated, see Propositon \ref{picfinite}. 

At the end, we want to mention that Theorem \ref{mainthm} answers a question for searching a new $2$-Fano manifolds\footnote{In the paper \cite{DS3}, de Jong and Starr introduce a new notion of higher Fano variety which generalizes Fano varieties, namely, $2$-Fano, cf. \cite{DS3} for the definition. Unfortunately, apart from\begin{enumerate}
\item low degree complete intersections in weighted projective spaces,
 \item some Grassmannians,
\item some hypersurfaces in some Grassmannians,
\end{enumerate}
one does not know other examples of 2-Fano varieties.} . In fact, \cite[Theorem 1.4]{CA} predicts that moduli spaces of rational curves on $n$-Fano complete intersections should be $(n-1)$-Fano. Therefore, people expect that one can find some new $2$-Fano manifolds among these "abstract" moduli spaces $\F$ and $\F_t$. Some moduli spaces $\F$ are indeed $2$-Fano, e.g. when $m=3$ and $X$ is of low degree. However, our main theorem shows that $\F$ is either non-Fano or a complete intersection surprisingly.\\

 \textbf{Structure of Paper.} In Section \ref{s3}, we accumulate and prove some useful facts for moduli space $\F$. In Section \ref{sectiondegeneration}, we analyze the possible degeneration type of the stable maps parametrized by $\F$ and show that $\Phi$ is well-defined on $\F_t$, cf. Proposition \ref{degenerationtype} and Lemma \ref{lemmphift}. We also introduce some morphisms $\pi_{p_i}$ on $\F_t$ which are used later to show that the complete linear system $|\lambda|_{\F_t}|$ gives an embedding of $\F_t$.
 
 In Section \ref{sectionRNC}, we characterizes rational normal curves and use the results to show $|\lambda|_{\F_t}|$ separates the points of $\F_t$ in Section \ref{sectioncyclerelation}. And then we prove that $\F_t$ contains a subvariety $Y$ which is a smooth complete intersection in a projective space. Meanwhile, $Y$ is also a complete intersection of hyperplanes induced by $|\lambda|_{\F_t}|$, see Corollary \ref{corintersectype}. In Section \ref{sectionembd}, we use the deformation theory of stable maps and the morphisms $\pi_{p_i}$ to show $|\lambda|_{\F_t}|$ gives an embedding of $\F_t$ into a projective spaces, see Proposition \ref{propembed}.

  In Section \ref{sectionmainthm}, we apply a criterion (Proposition \ref{mainprop}) to the pair $(\F_t,Y)$, and show that $\F_t$ is also a complete intersection in a projective space by induction. In Section \ref{sectionapp}, we give some applications of our main theorem to the rational connectedness of moduli spaces of rational curves, the enumerative geometry and the Picard group of $\F$.

  \textbf{Acknowledgments.} I am very grateful to my advisor Prof.~A.~J.~de Jong for teaching me a lot of moduli techniques and answering me some silly questions. His charitable guidance and unceasing optimism led me through the proof of the contained results. I also really appreciate Prof.~J.~Starr's help for reading some parts of this paper. Without his suggestions, I can not find out some gaps. I thank Prof.~B.~Hassett, Prof.~C.~Liu and Prof.~M.~Fedorchuk for reading some parts of this paper and pointing out to me some references, my friends Dr.~D.~Li and Dr.~D.~Yang in Max Planck Institute for some discussions. Some parts of this paper were written when I was a graduate student in the department of mathematics of Columbia University, and a postdoc in Washington University in St.Louis and Max Planck Institute for Mathematics. I am very grateful to these departments for providing the comfortable enviroments.

\section{Preliminary}\label{s3}
Let us recall some general facts about the moduli spaces of stable maps and the deformation theory of rational curves. The main references for this section are \cite[Section 5 and 6]{DS}, \cite[Chapter II]{K} and \cite{FP}. Suppose that $Y$ is a smooth complex projective variety. 

\begin{lemm}\cite[Kol 96]{K}, \cite[Lemma 3.1]{DS} and \cite{FP} \label{lemmdimcount}
Let $f:C\rightarrow Y$ be a stable map of arithmetic genus zero.
\begin{itemize}
\item If every component of $C$ is contracted or free, then deformations of $f$ are unobstructed. Moreover, there exist deformations of $(C,f)$ smoothing all the nodes of $C$. A general such deformation is free.

\item If the stable map $f$ is parametrized by $\M_{0,k}(Y,\beta)$ and deformations of $f$ are unobstructed, then $\M_{0,k}(Y,\beta)$ is smooth at $[(C,f,x_1,\ldots,x_k)]$ with dimension \[dim(\M_{0,k}(Y,\beta)_{[(C,f,x_1,\ldots,x_k)]})=dim(Y)+\int\limits_{\beta} c_1(T_Y)+k-3.\]

\end{itemize}

\end{lemm}

\begin{lemm}\cite[Lemma 5.1]{DS} \label{lemm00}
With the same notations as above, if every point in a general fiber of the evaluation map \[ev_{\beta}:\M_{0,m}(Y,\beta)\rightarrow Y^m\]
parametrizes a curve whose
irreducible components are all free, then a (non-empty) general fiber of $ev_{\beta}$ is smooth
of the expected dimension
\[ \int\limits_{\beta} c_1(T_Y)-(m-1)dim(Y)+m-3\]
and the intersection with the boundary is a simple normal crossings divisor.
\end{lemm}

\begin{lemm}\label{lemm0}
With the conditions \ref{forget}, the general fiber $\F$ is a smooth projective variety of the expected dimension \[(c+2-\sum\limits_{i=1}^c d_i)m+n-c-3\]with the boundary $\Delta$ which is a simple normal crossing divisor.
\end{lemm}

\begin{proof}
By \cite[Lemma 5.5 and Corollary 5.11]{DS}, the conditions \ref{forget} verify that $m\alpha$ is m-dominant and m-minimal curve class, see \cite[Definition 5.2]{DS}. It follows from \cite[Lemma 5.3]{DS} that $\F$ is a smooth projective variety of the expected dimension
\begin{align*}
\int\limits_{m\alpha} c_1(T_X)-(m-1)dim(X)+m-3=(c+2-\sum\limits_{i=1}^c d_i)m+n-c-3.
\end{align*}
with the boundary $\Delta$ which is a simple normal crossing divisor.

\end{proof}
The following lemma is well-known.
\begin{lemm}\label{lemcha3}
 If the dimension of $X$ is positive, then
\[H^0(X,\OO_X(1))=H^0(\PP^n,\OO_{\PP^n}(1))=n+1.\]
\end{lemm}



\begin{cor}\label{corcha3}
The smooth projective variety X is linearly non-degenerated, i.e., the variety X is not in any hyperplane.
\end{cor}

\begin{proof}
It follows from Lemma \ref{lemcha3} that $H^0(X,I_X(1))=0$
where $I_X$ is the ideal sheaf of $X$. It means that no linear form is vanishing on $X$, i.e., the variety $X$ is not in any hyperplane.
\end{proof}

\begin{cor} \label{generalpoints}
There are $m(\leq n+1)$ general points in $X\subseteq \PP^n$ such that they span a projective space $\PP^{m-1}$.
\end{cor}

\begin{proof}
We pick up a general point in $X$ first. Secondly, we pick up the second general point and get a line by spanning those two points. By Corollary \ref{corcha3}, we can pick up the third general point which is not on this line. Note that $n+1 \geq m$. We can process in this way to produce $m$ general points of $X$ which span $\PP^{m-1}$ by Corollary \ref{corcha3}.
\end{proof}



\begin{lemm}\label{defphi} \label{exception} \label{lemmembed}
The map $\Phi$ (\ref{phi}) is a morphism with the only exceptions
\begin{enumerate}
  \item $c=2, d_1=d_2=2, m\geq 6$,
  \item $c=1, d_1=3, m\geq 5$.
\end{enumerate}

\end{lemm}

\begin{proof}
 The lemma from the proof of \cite[Lemma 6.4]{DS}. We sketch the proof here for the sake of completeness. In fact, by Lemma \ref{lemm0}, one can verify \cite[Hypothesis 6.3]{DS} if the tuple $(c,d_1,\ldots,d_c,m)$ satisfies neither (1) nor (2). The map $\Phi$ is induced by the morphism \[\phi_{\lambda}: U_p \rightarrow \PP^{n-m}=\PP^{n}/\Span(p_1,\ldots, p_m)\]in the proof of \cite[Lemma 6.4]{DS}, where $U_p$ is the maximal open substack of the corresponding fiber of \[ev:\M_{0,m}(\PP^n,m)\rightarrow (\PP^n)^m\]parametrizing stable maps for which no non-contracted irreducible component is mapped into the
linear subspace $\Span(p_1, \ldots , p_m)$, see \cite[Definition 6.1]{DS} for more details. Note that $U_p$ contains $\F$ under the assumption, cf. \cite[Lemma 6.2]{DS}. In particular, we have that $\Phi=\phi_{\lambda}|_{\F}$. We prove the lemma.

\end{proof}

\begin{lemm} \label{lemmsection}
  Let $(C,f,x_1,\ldots,x_m)$ be a stable map parametrized by a point in $\F$. Suppose the domain $C$ consists of a comb with $m$ rational teeth. If the map $f$ collapses the handle and maps the teeth to the lines meeting at a point in $X$ (for the notations, see \cite[page 156, Definition 7.7]{K}), then the forgetful morphism $F:\F\rightarrow \M_{0,m}$ has a section $\sigma$.
\end{lemm}
We say that this map $f$ is of maximal degeneration type.
\begin{proof}
We describe the section $\sigma$ pointwise. It follows from the assumption that there are $m$ lines $l_1,\ldots,l_m$ in $X$ passing through $m$ general points $p_1,\ldots,p_m$ and intersecting at a different point $q\in X$. Suppose that $p_i\in l_i$. Let $(C,x_1,\ldots,x_m)$ be a point in $\M_{0,m}$. By \cite[Theorem 3.6]{BM}, there exists a stable map as follows:
\[ f: C\cup l_1\cup \ldots \cup l_m \rightarrow X\]
where the domain $C\cup l_1\cup \ldots \cup l_m$ is the union of $C$ and $l_i$ identifying $x_i$ with $p_i$ and the map $f$ is the identity on each $l_i$ and collapses $C$ to the point $q$. The section $\sigma$ associates to a point $(C,x_1,\ldots,x_m)$ the point $(C\cup l_1\cup \ldots \cup l_m,f,x_1,\ldots,x_m)$ in $\F$. It is easy to check it is a section of $\F$.

\end{proof}

\begin{cor}
With the same hypothesis as Lemma \ref{lemmsection}, the fiber $\F_t$ of the forgetful map $F$ over a general point $t\in \M_{0,m}$ is connected. Moreover, the fiber $\F_t$ is a smooth projective variety.
\end{cor}

\begin{proof}
By the Stein factorization \cite{H} and Lemma \ref{lemmsection}, we have the following factorization,
 \[
 \xymatrix{
 \F\ar[dr]^s \ar[dd]^F\\
 & Y\ar[dl]^h\\
 \M_{0,m}\ar@/^2pc/[uu]_{\sigma}
 }
 \]
where $Y$ is a normal variety, the morphism $h$ is finite over $\M_{0,m}$, the section $\sigma$ is a section of $F$ in Lemma \ref{lemmsection} and the fibers of $s$ are connected. Therefore, the composition of $\sigma$ and $s$ is a section of $h$. Since $h$ is a finite morphism between two normal varieties, the composition $s \circ \sigma$ is an isomorphism, i.e., the variety $Y$ is isomorphic to $\M_{0,m}$. It follows the first assertion. By the generic smoothness theorem, we have the second assertion.
\end{proof}

\section{Classification of Degeneration Type of Stable Maps} \label{sectiondegeneration}
In this section, we prove Lemma \ref{degree} which characterizes rational normal curves. And we use it to show the rational curves on $X$ passing through $l$ general points have degree at least $l$, cf. Lemma \ref{degreecomp} . This property of rational curves on $X$ shows that the image of a stable map parametrized by $\F_t$ is the union of lines and a rational normal curves, see Proposition \ref{degenerationtype}. It gives a classification of degeneration types of stable maps parametrized by $\F_t$. By this classification, we are able to show that $\Phi$ is well-defined on $\F_t$, cf. Lemma \ref{lemmphift}. At the end of this section, we introduce some morphisms $\pi_{p_i}$ on $\F_t$ which will be used later.

\begin{lemm}\label{degree} Let $C$ be a reduced and irreducible curve $C$ in $\PP^n$ passing through $\{p_1,\ldots,p_m\}$. Then, the degree of $C$ is at least $m-1$. The degree of $C$ is $m-1$ if and only if the curve $C$ is a rational normal curve in $\Span(p_1,\ldots,p_m)=\PP^{m-1}$.

\end{lemm}
\begin{proof}
Suppose that the degree of $C$ is at most $m-2$. The points $\{p_1,\ldots,p_{m-1}\}$ span a $\PP^{m-2}$ and the intersection of $\PP^{m-2}$ and $C$ contains these $m-1$ points. Therefore, the curve $C$ is in $\PP^{m-2}$. However, the point $p_m \in C$ is not in $\Span(p_1,\ldots,p_{m-1})$. It is a contradiction.

Note that $\Span(p_1,\ldots,p_m) \cap C$ contains $m$ points. By the intersection theory, the curve $C$ is in $\Span(p_1,\ldots,p_m)=\PP^{m-1}$ if the degree of $C$ is $m-1$. Moreover, the curve $C$ is nondegenerate in $\Span(p_1,\ldots,p_m)=\PP^{m-1}$. Hence, by \cite[Page 179]{GH1}, the curve $C$ is a rational normal curve in $\Span(p_1,\ldots,p_m)$.

\end{proof}

Let $(C,f,x_1,\ldots,x_m)$ be a stable map parametrized by $\F$. Assume that the image $f(C)$ in $X$ consists of $k$ irreducible components $C_1,\ldots,C_k$ with $f(C)=C_1 \cup \ldots \cup C_k$. Denote by $P$ the set $\{p_1,\ldots,p_m\}.$
\begin{defi}
With the notations as above, the subset $P(C_i)$ of $P$ is defined to be
$P\cap C_i.$ We say that $f$ is a pseudo-embedding if the restriction of $f:C\rightarrow X$ to each irreducible component of $C$ is a contraction or an embedding. 
\end{defi}
\begin{lemm} \label{degreecomp}
With the notations as above, we have that a rational curve on $X$ passing through $l$ general points has degree at least $l$.
In particular, every irreducible component $C_i$ of $f(C)$ has degree at least $\#(P(C_i))$.

\end{lemm}
\begin{proof}
In order to show the first assertion, we consider the evaluation map
\[ev_l: \M_{0,l}(X,(l-1)\alpha)\rightarrow X^l\]
associating a point $(C',g,y_1,\ldots,y_l)$ to the point $(g(y_1),\ldots, g(y_l))$. We calculate the dimension of $\M_{0,l}(X,(l-1)\alpha)$ at the point $(\PP^1,g,y_1,\ldots,y_l)$ parametrizing (irreducible) a rational curve on $X$ passing through a general point of $X$. It follows from Lemma \ref{lemmdimcount} that
\begin{align*}
dim(\M_{0,l}(X,(l-1)\alpha)_{(\PP^1,f,x_1,\ldots,x_l)})&=dim(X)+\int\limits_{(l-1)\alpha} c_1(T_X)+l-3\\
&=dim(X)+(l-1)(n+1-\sum\limits_{i=1}^c d_i)+l-3.
\end{align*}
Therefore, we have that 
\begin{align*}
dim(\M_{0,l}(X,(l-1)\alpha)_{(\PP^1,f,x_1,\ldots,x_l)})-dim(X^l)=(c+1-\sum\limits_{i=1}^c d_i)(l-1)+ l-3 < 0
\end{align*} 
by our setup in Section \ref{notation}. It follows that the evaluation map $ev_l$ is not dominant at $(\PP^1,g,y_1,\ldots,y_l)$. It implies the first assertion by Lemma \ref{degree}. 

The first assertion implies the second one.

\end{proof}

\begin{lemm} \label{rnc}
With the notations as above, the stable map $(C,f,x_1,\ldots,x_m)$ for which $\Span(f(C))$ is a projective space of dimension $m$ is a pseudo-embbedding map. Moreover, the image $f(C)(\subseteq X)$ is the union $C_1 \cup \ldots \cup C_k$ of rational normal curves with $P=\bigcup\limits_{i=1}^k P(C_i)  $, $P(C_i)\cap P(C_j)=\emptyset$ for $i\neq j$ and $deg (C_i)= \#P(C_i)$.
\end{lemm}
\begin{proof}
By Lemma \ref{degreecomp}, we have that
\[f(C)=C_1\cup\ldots\cup C_k\]
where $C_i$ is a reduced and irreducible curve on $X$ passing through $P(C_i)$ and the degree $deg(C_i)$ of $C_i$ is at least $\#(P(C_i))$. It implies that \begin{equation}\label{degineq1}
m=deg(f_*([C]))\geq deg(f(C)) = \sum\limits_{i=1}^k deg(C_i)\geq \sum\limits_{i=1}^k \#(P(C_i))
\end{equation}
where the first inequality is an equality if and only if $f$ is a contraction or an one-to-one map on each irreducible component of $C$, the second inequality is an equality if and only if $deg (C_i)= \#P(C_i)$.
On the other hand, it is clear that
\begin{equation}\label{degineq2}
P=\bigcup\limits_{i=1}^k P(C_i) \text{~and~} m=\#(P)\leq \sum\limits_{i=1}^k  \#(P(C_i))
\end{equation}
where the inequality is an equality if and only if the family of sets $P(C_i)$ is a partition of $P$, i.e., $P(C_i)\cap P(C_j)=\emptyset$ for $i\neq j$. 

From the argument above, all the inequalities in (\ref{degineq1}) and (\ref{degineq2}) are equalities. If all the $C_i$ are smooth rational normal curves, then $f$ is a pseudo-embbedding map and the proposition follows. Indeed, note that the irreducible and reduced curve $C_i$ passes through $P(C_i)$ which are in general position and the degree of $C_i$ is $\#(P(C_i))$. It follows that $\Span(C_i)$ is a projective space of dimension at most $deg(C_i)$. If the dimension of $\Span(C_i)$ attains $deg(C_i)$, then $C_i$ is a nondegenerate curve in $\Span(C_i)$ and it is a smooth rational normal curve by the argument in \cite[Page 179]{GH1}. If the dimension of $\Span(C_j)$ does not attain $deg(C_j)$ for some j, then the dimension of $\Span(f(C))=\Span(C_1\cup \ldots\cup C_k)$ does not attain the maximal value $\sum\limits_{i=1}^k deg(C_i)=m$. Therefore, it is a contradiction.

\end{proof}

\begin{lemm} \label{lemmspan}
If $(C ,f ,x_1,\ldots ,x_m)$ is parametrized by $\F_t$, then $\Span(f(C))$ is a projective space of dimension $m$.
\end{lemm}

\begin{proof}
It follows from Lemma \ref{defphi} that it suffices to prove the dimension of the locus in $\F$ parametrizing the stable mapped into $\Span(p_1,\ldots, p_m)$ is at most $m-4~(=dim(\M_{0,m})-1)$ for $(c,d_1,\ldots,d_c)=(1,3)$ and $(2,2,2)$. 

We follow the argument in \cite[Page 31]{DS}. Let $Z$ be the intersection of $X$ with a general $\PP^{m-1}$ in $\PP^n$. Then we have that
\[dim(Z)=dim(X)+m-h^0(X,\OO_X(1))\]
and (by the adjunction formula) \[\int\limits_{\alpha}c_1(T_Z)=\int\limits_{\alpha}c_1(T_X)+m-h^0(X,\OO_X(1)).\]
We may assume that the general $\PP^{m-1}$ is $\Span(p_1,\ldots,p_m)$ and $\{p_1,\ldots ,p_m\}$ are general points of $Z$ and $X$. It is clear that the evalution map \[ev_Z:\M_{0,m}(Z,m\alpha)\rightarrow Z^m\] is domaint, otherwise, there are no stable maps parametrized by $\F$ and mapped into $\Span(p_1,\ldots,p_m)$, which implies that $\Phi$ has no indeterminate locus. Moreover, the class $m\alpha$ is $m$-minimal-dominant for $X$ in the sense of \cite[Definition 5.2]{DS} by \cite[Lemma 5.5]{DS}, and the general fiber $\F_Z$ of $ev_Z$ over $\{p_1,\ldots,p_m\}$ is contained in $\F$. It follows that the class $m\alpha$ is $m$-minimal-dominant for $Z$. Therefore, we can apply \cite[Lemma 5.3]{DS} and Lemma \ref{lemm00} to $Z$. It follows that the dimension of the general fiber $\F_Z$ is given by 
\begin{align*}
&\int\limits_{m\alpha} c_1(T_Z)-(m-1)dim(Z)+m-3 \\
=&m(n+1-\sum\limits_{i=1}^c d_i)+(m-h^0(X,\OO_X(1)))-(m-1)dim(X)+m-3.
\end{align*}
A direct calculation show that $dim(\F_Z)$ is at most $m-4$ for $(c,d_1,\ldots,d_c)=(1,3)$ and $(2,2,2)$. It is clear that $\F_Z$ is the locus in $\F$  parametrizing the stable mapped into $\Span(p_1,\ldots, p_m)$. We prove the lemma.
\end{proof}

\begin{prop} \label{degenerationtype}
 Let $(C ,f ,x_1,\ldots,x_m)$ be a stable map parametrized by $\F_t$ (see \ref{forget}). Then,
 \begin{enumerate}
   \item  the curve $C$ consists of $C_1$ and $l_1,\ldots,l_k$ such that $C$ is a comb with handle $C_1$ and teeth $\{l_i\}$, moreover, there is a unique pointed point $x_{j_i}$ is on $\{l_i\}$ and the rest pointed points among the points \{$x_1,\ldots,x_m$\} are on $C_1$, 
   \item the map $f$ is a pseudo-embedding, the image $f(l_i)$ of $l_i$ is a line in $X$ passing through $p_{j_i}$ and $f(C_1)$ is a rational normal curve of degree $m-k$ for $0\leq k \leq m$. In particular, if $k\neq m$, then $(C ,f ,x_1,\ldots,x_m)\in \F_t$ is an embedded curve.
 \end{enumerate}

\end{prop}

\begin{remark}\label{rmkmaxdeg}
For the case $k=m$, the map $f$ collapses the handle $C_1$ to a point $p$ and maps the tooth $l_i$ isomorphically to a line passing through $p$ and $p_i$.
\end{remark}

\begin{proof}
Suppose that $C$ consists of rational curves $C_1,l_1,l_2,\ldots,l_s$. Since a point $t\in \M_{0,m}$ is general, the image $F([C ,f ,x_1,\ldots,x_m])$ parametrizes a smooth rational point with $m$ pointed points. Hence, there is a unique component of $C$ which is not contracted by the stabilization process of the forgetful map $F$, say $C_1$. Let $D$ be a connected component of the union of $l_1,\ldots,l_s$. 

We claim that there is a unique pointed point on $D$. In fact, note that the stabilization process of $F$ contracts $D$ to a point. It follows that $D$ has at most one pointed point. If $D$ has no any pointed point, then the map $f$ contracts $D$ to a point by Lemma \ref{rnc} and Lemma \ref{lemmspan}. Since the stabilization process of $F$ contracts $D$, the curve $C$ has infinitely many automorphisms mapping $D$ into $D$ and fixing $C-D$. Therefore, these automorphisms of $C$ induce infinitely many automorphisms of the stable map $f:C\rightarrow X$. It contradicts with the stability of $f$. We prove the claim. 

A similar discussion implies that $D$ has a unique irreducible component. Namely, $D$ consists of exactly one of $l_1, \ldots, l_s.$ In fact, suppose that $D$ has at least two irreducible components. By the claim above, there is one irreducible component $l$ of $D$ which does not contain any pointed point. It follows from Lemma \ref{rnc} that the map $f$ contracts $l$ to a point. Since the stabilization process of $F$ contracts $l$, the curve $C$ has infinitely many automorphisms mapping $l$ into $l$ and fixing $C-l$. As above, there are infinitely may automorphisms of $f:C\rightarrow X$ induced by these automorphisms of $C$. It contradicts with the stability of $f$. We show the first assertion.

The second assertion follows from the first one and Lemma \ref{lemmspan} and Lemma \ref{rnc}.
\end{proof}

\begin{lemm} \label{lemmphift} \label{remark2}
Let $(C ,f ,x_1,\ldots,x_m)$ be a stable map parametrized by $\F_t$. Then $f(C)$ is smooth at $p_i$ and the tangent direction $T_{f(C),p_i}$ points out of $\Span(p_1,\ldots,p_m).$
In particular, the restriction map $\Phi|_{\F_t}$ of the map $\Phi$  is a morphism.

\end{lemm}
\begin{proof}
We follow the proof of Lemma \ref{defphi}. We use the notations of Proposition \ref{degenerationtype}. If $f(C_1)$ is in $\Span(p_1,\ldots, p_m)$, then $f(l_i)$ are in $\Span(p_1,\ldots, p_m)$, therefore, $f(C)$ is in $\Span(p_1,\ldots, p_m)$ which contradicts with Lemma \ref{lemmspan}. If $f(l_1)$ is in $\Span(p_1,\ldots, p_m)$, then the intersection of $\Span(p_1,\ldots, p_m)$ and $f(C_1)$ contains $m-k$ pointed points plus $f(l_1\cap C_1)$, therefore, $f(C_1)$ is in $\Span(p_1,\ldots, p_m)$ by the fact that $f(C_1)$ is a rational normal curve of degree $m-k$, see Proposition \ref{degenerationtype}. As before, it is also a contradiction. In summary, no non-contracted irreducible component of $C$ is mapped into the
linear subspace $\Span(p_1, \ldots , p_m)$. By Proposition \ref{degenerationtype}, if the tangent direction $T_{f(C),p_i}$ are in $\Span(p_1,\ldots,p_m)$, then the component of $f(C)$ passing through $p_i$ is in $\Span(p_1,\ldots,p_m)$ which is a contradiction. We show the first assertion.

It is clear that $U_p$ (see \cite[Definition 6.1]{DS}) contains $[(C ,f ,x_1,\ldots,x_m)]$, hence, $\F_t\subseteq U_p$. It implies the map $\Phi$ is well-defined on $\F_t$, see the proof of Lemma \ref{defphi}.
\end{proof}

Now, we are able to define some maps on $\F$ and $\F_t$ which will be used in the following sections. Suppose that $\F$ (resp. $\F_t$) has the universal family $(\pi: \U\rightarrow \F, f_{\U}:\U\rightarrow X,  \sigma_1, \ldots, \sigma_m)$ (resp. $(\pi_t: \U_t\rightarrow \F, f_{\U_t}:\U_t\rightarrow X,  \sigma_1, \ldots, \sigma_m)$) where $\sigma_i$ are disjoint sections of $\pi$. Note that we have a diagram as follows
\begin{equation}\label{eq:section}
\xymatrix{\U \ar[r]^{f_{\U}} \ar[d]^{\pi} &X\\
\F\ar@/^2pc/[u]^{\sigma_i}}
\end{equation}
and it induces a map $f_{\U}^*\Omega_X^1\rightarrow \Omega_{\U}^1$ where $\U$ is the universal bundle over $\F$ and $\sigma_i$ is the universal section induced by the i-th pointed point of $\U$. Since there is a canonical map $\Omega^1_{\U}\rightarrow \Omega_{\U/\F}^1$, the composition of these two maps gives rise to $f_{\U}^*\Omega^1_X\rightarrow \Omega_{\U/\F}^1$. It induces a morphism
\[\sigma_i^*f_{\U}^*\Omega_X^1\rightarrow\sigma_i^*\Omega_{\U/\F}^1.\]
Since the image of $\sigma_i$ is in the locus of the smooth point of $\pi$ and the composition $f_{\U}\circ\sigma_i$ is a constant map with value $p_i$, it gives rise to a map
\[(T_{p_i}X)^{\vee}\otimes_{\mathbb{C}}\OO_\F \rightarrow \sigma_i^*\omega_{\U/\F}\]
where $\omega_{\U/\F}$ is the dualizing sheaf of $\pi$. The map is surjective at a point \[[(C,f,x_1,\ldots,x_m)]\in \F\] if and only if $df_*(T_{C,x_i})$ is not zero. It induces a map as follows.

\begin{defi}\label{defpi}
We define a map
\[\pi_{p_i}:\F \dashrightarrow \PP(T_{p_i}X)=\PP^{n-c-1}\]
associating to a point $(C,f, x_1,\ldots,x_m)\in \F$ the tangent direction  \[T_{f(C),p_i}=df_*(T_{C,x_i}) \in \PP(T_{p_i}X),\] where the point $p_i$ is $f(x_i)$.
\end{defi}
The locus of indeterminacy of $\pi_{p_i}$ parametrizes stable maps $(C,f, x_1,\ldots,x_m)$ such that $df_*(T_{C,x_i})=0$. 
Similarly, we have a map \[h:(T_{p_i}X)^{\vee}\otimes_{\mathbb{C}}\OO_\F \rightarrow \sigma_i^*\omega_{\U_t/\F_t}.\]
By Proposition \ref{degenerationtype}, the map $h$ is surjective and induces a morphism $\pi_{p_i}|_{\F_t}$ which is the restriction of the map $\pi_{p_i}$ to $\F_t$. In particular, the rational map $\pi_{p_i}$ is well-defined on $\F_t$.

\begin{lemm} \label{sublinear}
we have that \[(\pi_{p_i}|_{\F_t})^*(\OO_{\PP^{n-c-1}}(1))\simeq (\Phi|_{\F_t})^*(\OO_{\PP^{n-m}}(1)).\]
\end{lemm}
\begin{proof}
The statement is clear from the following commutative diagram:
\begin{equation} \label{projection}
\xymatrix{
\F_t \ar@/_1pc/[dr]^{\Phi|_{\F_t}} \ar[r]^{\pi_{p_i}|_{\F_t}} & \PP(T_{p_i}X)\ar@{.>}[d]^L \\
&\PP^n/\Span(p_1,\ldots,p_m)
}
\end{equation}
where L is the projection map 
\[\xymatrix{L :\PP(T_{p_i}X) \ar@{}[d]^{\|} \ar@{.>}[r]& \PP^n/\Span(p_1,\ldots,p_m) \ar@{}[d]^{\|}\\
\PP^{n-c-1}\ar@{.>}[r]^{pr} & \PP^{n-m}
}\]
associating to a point $[\overrightarrow{v}]\in \PP(T_{p_i}X)$ the point (in $\PP^{n-m}$)  parametrizing \[\Span(\overrightarrow{v},p_1,\ldots,p_m)/\Span(p_1,\ldots,p_m).\]
The indeterminate locus of $L$ consists of the points in $\PP(T_{p_i}X)$ which can be represented by a nonzero vector $\overrightarrow{v}\in T_{p_i}X$ lying in $\Span(p_1,\ldots,p_m)$, hence, the image of $\pi_{p_i}$ is outside of the indeterminate locus of $L$ by Lemma \ref{remark2}. The composition $L\circ \pi_{p_i}|_{\F_t}$ is a morphism and it is $\Phi|_{\F_t}$.

\end{proof}

\section{Geometry of Rational Normal Curves} \label{sectionRNC}
In this section, we analyze rational normal curves in a projective spaces. We show that two rational normal curves of degree $n$ coincide if they pass through the same $n$ points in general position with the same tangent directions at these points. The way to show this property is to put these two rational normal curves in a smooth surfaces and calculate their intersection, see Proposition \ref{propsametang}. We show that the degeneration types (cf. Definition \ref{degtypes}) of two stable maps parametrized by $\F_t$ are the same if they have the same image for all $\pi_{p_i}$, see Lemma \ref{lemm2}.

\begin{defi}
Suppose that $C$ is a smooth rational curve in $\PP^n$ with 
 normal bundle
\[N_{C/\PP^n}=\bigoplus\limits ^{n-1}_{i=1} \OO_C(a_i)\]
and $a_1\leq a_2 \leq \ldots \leq a_{n-1}$. We say that the rational curve $C$ is almost balanced if $a_{n-1}-a_1\leq 1$.
\end{defi}

In the paper \cite{Z}, Z.~Ran gives a careful analysis about the balanced property of rational curves in the projective space. We cite one result from the paper \cite{Z}.

\begin{thm}\cite[Theorem 6.1 (Sacchiero)]{Z}\label{thmrationalbal}
A general rational curve of degree $d\geq n$ in $\PP^n$ is almost balanced.
\end{thm}

\begin{cor} \label{normalbundle}
If $C$ is a rational normal curve in $\PP^n$, then the normal bundle is
\[N_{C/\PP^n}=\bigoplus\limits ^{n-1}_{i=1} \OO_C(n+2).\]
\end{cor}

\begin{proof}
Since we have two short exact sequences as follows,
\[0\rightarrow T_C \rightarrow T_{\PP^n}|_C \rightarrow N_{C/\PP^n}=\bigoplus\limits ^{n-1}_{i=1} \OO_C(a_i)\rightarrow 0\]

\[0\rightarrow \OO_{C}\rightarrow \bigoplus\limits ^{n+1}_{i=1} \OO_{\PP^n}(1)|_C=\bigoplus\limits ^{n+1}_{i=1} \OO_C(n)\rightarrow T_{\PP^n}|_C\rightarrow 0\]
they give rise to the following equalities,

\begin{enumerate}
  \item $\deg N_{C/\PP^n}=\sum\limits_{i=1}^{n-1} a_i$,
  \item $\deg N_{C/\PP^n}=\deg(T_{\PP^n}|_C)-\deg(T_C)=(n-1)(n+2)$.
\end{enumerate}
Suppose that $a_1\leq a_2\ldots \leq a_{n-1}$. Note that every rational normal curve in $\PP^n$ is projectively linearly isomorphic to each other. By Theorem \ref{thmrationalbal}, the rational normal curves are almost balanced $a_{n-1}-a_{1}\leq 1$. Hence, we have $a_1=a_2=\ldots=a_{n-1}=n+2$.
\end{proof}
The following corollary will be only used in the proof of Proposition \ref{propembed}.
\begin{cor} \label{corovanish}
If the curve $f:C\hookrightarrow X$ is a rational normal curve of degree $m\geq 3$ passing through $p_1,\ldots p_m$, then $H^0(C, N_{C/X}(\sum\limits_{i=1}^m -2p_i))=0$.
\end{cor}

\begin{proof}
In fact, we have a short exact sequence
\[0\rightarrow N_{C/X}\rightarrow N_{C/\PP^n} \rightarrow N_{X/\PP^n}|_C \rightarrow 0\]
for the smooth rational normal curve $C$. It follows that
\[H^0(C, N_{C/X}(\sum\limits_{i=1}^m -2p_i))\subseteq H^0(C, N_{C/\PP^n}(\sum\limits^m_{i=1} -2p_i))=H^0(C, N_{C/\PP^n}(-2m)).\]
The exact sequence of normal bundles $
0\rightarrow N_{C/\PP^m}\rightarrow N_{C/\PP^n}\rightarrow N_{\PP^m/\PP^n}|_C\rightarrow 0$
implies that $N_{\PP^m/\PP^n}|_C(-2m)=\bigoplus\limits^{n-m}\OO_{\PP^{n-m}}(-m)$ by $C \subseteq \PP^m \simeq \Span(C)$ and $N_{\PP^m/\PP^n}=\bigoplus\limits^{n-m}\OO_{\PP^{n-m}}(1)$. Therefore, we have $$H^0(C, N_{C/\PP^n}(-2m))=H^0(C, N_{C/\PP^m}(-2m)).$$
It follows from Corollary \ref{normalbundle} and $m\geq 3$ that $$H^0(C, N_{C/\PP^m}(-2m))=H^0(C,\bigoplus\limits^{m-1}_{i=1}\OO_C(2-m))=0.$$ We prove the corollary.
\end{proof}

\begin{prop}\label{propsametang}
Suppose that two curves $C$ and $C'$ are rational normal curves in $\PP^n$. Let $\{p_1,\ldots,p_n\}$ be $n$ points of $\PP^n$ in general position. If $C$ and $C'$ satisfy the following assumptions
\begin{enumerate}
  \item both curves $C$ and $C'$ pass through the points $\{p_1,\ldots,p_n\}$ and
  \item $T_{p_i}C=T_{p_i}C'$ for all $i$,
\end{enumerate}
where $T_{p_i}C$ (resp. $T_{p_i}C'$) is the tangent line to $C$ (resp. $C'$) at point $p_i$, then $C=C'$.
\end{prop}

\begin{proof}
We show the proposition by the induction on $n$. Take $n=3$ first. We project $C$ and $C'$ from $p_1$ to a plane $\PP^2$ and denote this projection by $\pi$. Hence, the image $\pi(C)$ and $\pi(C')$ are conics passing through the point $Q=\pi(p_1)=T_{p_1}C \cap \PP^2$. The conics are tangent to two distinct lines $l_1=\pi(T_{p_2}C)$ and $l_2=\pi(T_{p_3}C)$. It is obvious that $Q$ is neither on these two lines. By an elementary calculation, there is a unique conic $S$ passing through $Q$ and with tangent lines $l_1$ and $l_2$. In fact, any conic which is tangent to ${x_0=0}$ and ${x_1=0}$ is defined by the equation
\[ax_2^2+bx_0x_1=0\]
where $a$ and $b$ are determined by the coordinates of $Q$. By the discussion above, we know that $C,C'\subseteq \Cone (S)$ where $\Cone (S)$ is the projective cone of $S$ with the singular point $p_1$. For the point $p_2$, the same argument implies that $C$ and $C'$ are in another singular quadric $\Cone (S')$ with the singular point $p_2$. It follows that
\[C\cup C' \subseteq \Cone (S')\cap \Cone (S).\]
Since the intersection of $\Cone (S')$ and $\Cone (S)$ is a curve of degree $4$, we have that $C=C'$.

Suppose that the proposition holds for $n=s\geq3$. When $n=s+1$, as above
, we project the curves $C,C'$ from the point $p_1$ to a hyperplane $\PP^{s}$. The projection is denoted by $\pi$. It follows that $\pi(C),\pi(C')\subseteq \PP^s$ are rational normal curves of degree $s$ with the same tangent line $\pi(T_{p_i}C)=\pi(T_{p_i}C')$ at $\pi(p_i)$ for $i=2,3,\ldots,s+1$. By the induction, we conclude that $\pi(C)=\pi(C')=D$. Hence, both $C$ and $C'$ are in $\Cone (D)$ joining $p_1$ and $D$.

We blow up the unique singular point $p_1$ of $\Cone (D)$. By \cite[Chapter IV]{H}, we have
\[
\xymatrix{ \Bl_{p_1}(\Cone (D))\ar@{}[r]|-=&\PP(\OO\oplus\OO(-s))\ar[d]^{pr} \ar[rr]^{|B+sf|}_{Bl} &&\Cone (D)\subseteq \PP^{s+1}\\
&\PP^1 \ar@/^2pc/[u]^{\sigma}
}
\]
where $B$ is the curve class of section $\sigma(\PP^1)$ and $f$ is a fiber of $pr$. The section is unique among the curve classes such that the self-intersection is $-s$. 

We claim that
\[ [\widetilde{C}]=[\widetilde{C'}]=B+(s+1)f \in \Pic(\PP(\OO\oplus\OO(-s)))\]
where $\widetilde{C}$ and $\widetilde{C'}$ are proper transforms of $C$ and $C'$ respectively.
In fact, assume $\widetilde{C}=aB+bf$. Since $\Bl_*([\widetilde{C}])=[C]$, we have that
\[\widetilde{C}\cdot(B+sf)=b=\Bl_*([\widetilde{C}])\cdot\OO_{\PP^{s+1}}(1)=s+1.\]
The second equation follows from the projection formula, see \cite[Chapter 8]{F}. It follows that $b=s+1$.
Obviously, the curves $B$ and $\widetilde{C}$ intersect transversally at a point. Therefore,
it implies that $B\cdot \widetilde{C}=-sa+s+1=1$.
The similar argument works for $\widetilde{C'}$. We show the claim.

It follows from the claim that\[\widetilde{C} \cdot \widetilde{C'}=(B+(s+1)f)^2=-s+2(s+1)=s+2.\]
Note that $\widetilde{C}$ and $\widetilde{C'}$ are tangent at $p_2,\ldots,p_{s+1}$ and intersect with $B$ at the same point. If $\widetilde{C}\neq\widetilde{C'}$, then we have that\[\widetilde{C} \cdot \widetilde{C'} \geq 2s+1 >s+2.\]
It is a contradiction. It follows that \[C=\Bl(\widetilde{C})=\Bl(\widetilde{C'})=C'.\]
\end{proof}

\begin{lemm} \label{generalposition}
If $C$ is a rational normal curve of degree $m$ in $\PP^m$, then any $m+1$ distinct points on $C$ are in general position.
\end{lemm}
\begin{proof}
It is a direct calculation on the standard rational normal curve, see \cite[Chapter I]{J}.
\end{proof}

\begin{lemm} \label{lemm1}
If $C$ is a rational normal curve of degree $m$ in $\PP^m$ passing through the points $p_1,\ldots,p_m$, then we have that
\[\Span (p_1,\ldots,p_m,T_{p_i}C)=\PP^m.\]
Suppose that $1<k<m$, $m\geq 2$. If 
\[\Span (Q_1,\ldots,Q_k, T_{Q_1}C,\ldots,T_{Q_k}C)=\PP^l\]
for distinct $k$ points $Q_1,\ldots,Q_k$ on $C$, then $l$ is greater than $k$.
\end{lemm}

\begin{proof}
Since $\Span(C)$ is a projective space $\PP^m$ and $\Span(p_1,\ldots,p_m,T_{p_i}C)$ is tangent to $C$, we have
\[\deg(C\cap \Span(p_1,\ldots,p_m,T_{p_i}C))>m.\]

If $\Span(p_1,\ldots,p_m,T_{p_i}C)=\PP^{m-1}$, then it is a contradiction with $\deg C=m$. We prove the first assertion.

Since $\Span (Q_1,\ldots,Q_k, T_{Q_1}C,\ldots,T_{Q_k}C)$ contains $k$ points and at least one tangent line of $C$, similar as above, we have \[\Span (Q_1,\ldots,Q_k, T_{Q_1}C,\ldots,T_{Q_k}C)=\PP^l\]
where $l\geq k$.

Suppose that $\Span (Q_1,\ldots,Q_k, T_{Q_1}C,\ldots,T_{Q_k}C)=\PP^k$. We can pick up $m-k-1$ points on $C$ but not in $\Span (Q_1,\ldots,Q_k, T_{Q_1}C,\ldots,T_{Q_k}C)=\PP^k$. By Lemma \ref{generalposition}, $\Span (Q_1,\ldots,Q_k, T_{Q_1}C,\ldots,T_{Q_k}C)$ with these $m-k-1$ points on $C$ spans a projective space $\PP^{m-1}$ of dimension $m-1$. Hence, we have
\[\deg(\PP^{m-1}\cap C)\geq m-k-1+2k=m+k-1>m.\]
It is a contradiction. It follows that $l$ is greater than $k$.
\end{proof}

\begin{defi} \label{degtypes}
We say that two stable maps 
\begin{center}
   $(C,f,x_1,\ldots,x_m)$ and $(C',f',x'_1,\ldots,x'_m)$
\end{center}
parametrized by $\F_t$ (see \ref{forget}) are of the same degeneration type if they satisfy the following property:
\begin{center}
$x_i$ is on the handle of $C$ iff $x'_i$ is on the handle of $C'$  
\end{center}
where a handle of a stable map parametrized by $\F_t$ is described in Proposition \ref{degenerationtype}.

\end{defi}
\begin{lemm} \label{lemm2}
Suppose that two stable maps \[(C,f,x_1,\ldots,x_m) \text{~and~} (C',f',x'_1,\ldots,x'_m)\] are parametrized by $\F_t$ $(m\geq3)$. If $$\pi_{p_i}([(C,f,x_1,\ldots,x_m)])=\pi_{p_i}([(C',f',x'_1,\ldots,x'_m)])$$
for all i, then stable maps $(C,f,x_1,\ldots,x_m) \text{~and~} (C',f',x'_1,\ldots,x'_m)$ have the same degeneration type.
\end{lemm}

\begin{proof}
The assumption \[\pi_{p_i}([(C,f,x_1,\ldots,x_m)])=\pi_{p_i}([(C',f',x'_1,\ldots,x'_m)])\]
just means $T_{p_i}f(C)=T_{p_i}f'(C')$, where $T_{p_i}f(C)$ is the tangent line to $f(C)$ at the point $p_i$, similarly, for $T_{p_i}f'(C')$, see Lemma \ref{remark2}.

By Proposition \ref{degenerationtype}, we can suppose that
\begin{center}
    $f(C)=C_1\cup L_1\cup\ldots\cup L_k$ and $f'(C')=C'_1\cup L'_1\cup\ldots \cup L'_s$
\end{center}
where $0\leq k,s\leq m$, the components $C_1,C'_1$ are rational normal curves and the components $\{L_i\}, \{L'_i\}$ are lines. We hope it will cause no confusion if we sometimes also denote $f(C)$ (resp. $f'(C')$) by $C$ (resp. $C'$). 

In the maximal degeneration case (cf. Remark \ref{rmkmaxdeg}), 
the curve $C_1$ or $C_2$ is collapsed to a point. The lemma in this case is trivial. 

Suppose that the curve $C$ is smooth. By the second assertion of Lemma \ref{lemm1}, the curve $C'$ is smooth, i.e., $s=0$. Hence, we can assume that the numbers $k,s$ are at least $1$. If both lines $L_1$ and $L'_1$ (after relabeling the teeth) are passing through the same pointed point $p_1$ with the tangent line $T_{p_1}L_1=T_{p_1}L'_1$, then we have that
\[L_1=L'_1.\]
We take out the $L_1=L'_1$ from these two stable maps. By induction on the number of the teeth (k and s), we can show the assertion.

Therefore, we can assume that the curve $C'_1$ contains the points \[p_1(\in L_1),\ldots,p_k(\in L_k).\] Note that the intersection $C_1\cap C'_1$ consists of $(m-s)-k(\geq 0)$ points which are in general position.

Suppose that the intersection $C_1\cap C'_1$ consists of the points \[\{p'_1,p'_2,\ldots ,p'_{m-s-k}\}(\subseteq \{p_1,\ldots,p_m\}).\] By Proposition \ref{degenerationtype}, we have that
\begin{center}
$\Span C_1=\PP^{m-k}$ and $\Span C'_1=\PP^{m-s}$.
\end{center}
Note that the curve $C'_1$ contains the points \{$p_1,\ldots,p_k$\}. It follows that \[\Span(C_1\cup C'_1)=\PP^m.\] Therefore, we conclude that
\[\Span C_1 \cap \Span C'_1= \PP^{(m-s)+(m-k)-m}=\PP^{m-s-k}.\]
By Lemma \ref{lemm1}, if $m-s-k$ is at least $2$, then
\[\Span(p'_1,\ldots,p'_{m-s-k}, T_{p'_1}C,\ldots,T_{p'_{m-s-k}}C )=\PP^l\]
and $l$ is greater than $m-s-k$.
However, we have \[\Span(p'_1,\ldots,p'_{m-s-k}, T_{p'_1}C,\ldots,T_{p'_{m-s-k}}C )\subseteq \Span C'_1\cap \Span C_1=\PP^{m-s-k}.\]
It is a contradiction. It follows that $m-s-k$ is at most $1$.

If $m-s-k=0$, then the intersection $\Span C_1\cap \Span C'_1$ is a point. Since the curve $C_1$ contains the points \{$p_1,\ldots,p_k$\}, the space $\Span C$ contains $L_1,\ldots, L_k$. Note that the lines $L_1,\ldots, L_k$ intersect $C_1$ at $k$ points. It follows that $k=1$. By the symmetry, we also have that $s=1$. Therefore, we conclude that $m=2$. It is a contradiction.

If $m-s-k=1$, then the intersection $\Span C_1 \cap \Span C'_1$ is $\PP^1$. It also implies that the intersection $C_1\cap C'_1$ is a point. We claim that the intersection of $\Span C'_1$ and $C_1$ consists of two points and one of them is $C_1\cap C'_1$.

In fact, we suppose that the intersection of $\Span C'_1$ and $C_1$ consists of at least 3 points. By Lemma \ref{generalposition}, the intersection $\Span C_1 \cap \Span C'_1$ contains a projective plane $\PP^2$. It is impossible. So the claim is clear.

Note that $L_1,\ldots,L_k (\subseteq \Span C'_1)$ intersect $C_1$ at distinct $k$ points $\{c_1,\ldots,c_k\}$. 
Since $C_1\cap C'_1 \cap \{c_1,\ldots,c_k\}$ is empty, we have $\{c_1,\ldots,c_k\}=\Span C'_1\cap C_1-C_1\cap C'_1$. It follows $k=1$. By the symmetry, it implies that $s=1$. Therefore, we conclude $m=1+k+s=3$. In this case, by some elementary analysis of the possible degeneration types, the two stable maps have the same degeneration type.

\end{proof}

\section{Cycle Relations} \label{sectioncyclerelation}  
Recall that the restriction $\Phi|_{\F_t}:\F_t\rightarrow \PP^{n-m}$ is a morphism, cf. Lemma \ref{lemmphift}. Denote $(\Phi|_{\F_t}) ^*(\OO_{\PP^{n-m}}(1))$ by $\lambda|_{\F_t}$. In this section, we show a cycle relation Proposition \ref{cyclerelation}. In particular, each irreducible component of the boundary divisor $\Delta_t$ of $\F_t$ is a smooth and ample divisor. Firstly, we use the results in Section \ref{sectionRNC} to show the complete linear system $|\lambda|_{\F_t}|$ separates points, cf. Proposition\ref{propseppts}.

\begin{prop} \label{propseppts} 
Let $(C,f,x_1,\ldots,x_m)$ and $(C',f',x'_1,\ldots,x'_m)$ be two stable maps parametrized by $\F_t$.
If $\pi_{p_i}([(C,f,x_1,\ldots,x_m)])=\pi_{p_i}([(C',f',x'_1,\ldots,x'_m)])$ for all $i$,
then \[(C,f,x_1,\ldots,x_m)=(C',f',x'_1,\ldots,x'_m).\]
The complete linear system $|\lambda|_{\F_t}|$ on $\F_t$ separates points. In other words, the morphism $|\lambda|_{\F_t}|:\F_t\rightarrow \PP^N$ induced by $|\lambda|_{\F_t}|$ is injective.
\end{prop}

\begin{proof}
Recall that the maps
\[\pi_{p_1}|_{\F_t},\pi_{p_2}|_{\F_t}\ldots,\pi_{p_m}|_{\F_t}\]
are induced sections of $\lambda|_{\F_t}$, see Lemma \ref{sublinear}. 
The second assertion follows from the first assertion. 

By Lemma \ref{lemm2}, the stable maps parametrized by \[(C,f,x_1,\ldots,x_m) \text{~and~} (C',f',x'_1,\ldots,x'_m)\] have the same degeneration type. So, it follows from Proposition \ref{propsametang} that $$(C,f,x_1,\ldots,x_m)=(C',f',x'_1,\ldots,x'_m)$$ if the curve $C$ or $C'$ is smooth. 

In general, by Proposition \ref{degenerationtype}, the image curves $f(C)$ and $f'(C')$ are the union of a rational normal curve and several lines. The configuration of these lines are uniquely determined by the degeneration type and the tangent directions at the pointed points. 
 By induction on the number of components of $C$ and $C'$, we show the first assertion.

\end{proof}

\begin{lemm}\label{divisorNCD}
The divisor $\Delta_t$ is a simple normal crossing divisor on $\F_t$.
\end{lemm}

\begin{proof}
Recall that the boundary divisor $\Delta$ of $\F$ is a simple normal crossings divisor in $\F$, see Lemma \ref{lemm0}. Since $t\in \M_{0,m}$ is a general point, the corollary follows from Lemma \ref{lemm0} and the generic smoothness theorem.
\end{proof}

\begin{defi} \label{definedivisor}

Let $\Delta$ be the boundary divisor of $\F$. Denote $\F_t \cap \Delta $ by $\Delta_t$. For $i\in \{1,2\ldots,m\}$, the divisor $\Delta_{i,t} (\subseteq \Delta_t)$ is the divisor whose general points are parameterizing the stable maps whose image is the union of a line containing $p_i$ and a rational normal curve, see Proposition \ref{degenerationtype}. For simplicity, we sometimes denote by $\Delta_i$ the divisor $\Delta_{i,t}$. 
\end{defi}

 Recall that $\F$ (resp. $\F_t$) has the universal family $(\pi: \U\rightarrow \F, f_{\U}:\U\rightarrow X,  \sigma_1, \ldots, \sigma_m)$ (resp. $(\pi_t: \U_t\rightarrow \F, f_{\U_t}:\U_t\rightarrow X,  \sigma_1, \ldots, \sigma_m)$) where $\sigma_i$ are disjoint sections of $\pi$. The key fact that we use to prove our main theorem is Proposition \ref{cyclerelation} which shows that $\Delta_i=\sigma_i^*\omega_{\U_t/\F_t}$. To show Proposition \ref{cyclerelation}, we reduce to prove that the equality $\Delta_i=\sigma_i^*\omega_{\U_t/\F_t}$ holds numerically, cf. Lemma \ref{numertri}. In the numerical case, we reduce the proof to a simple case in which $\F_t$ is a curve and $\U_t$ is a blow-up surface of $\F_t\times \PP^1$, see Lemma \ref{numertri} for details.

\begin{lemm}\label{lemmaI}
With the same notations as in the diagram (\ref{eq:section}) and Lemma \ref{sublinear}, we have that $\lambda|_{\F_t}=\sigma_i^*\omega_{\U_t/\F_t}$, cf. (\ref{forget}).
\end{lemm}

\begin{proof}
Since the image of $\sigma_i$ is in the smooth locus of $\pi_t$, we have that
\[\sigma _i^*\omega_{\U_t/\F_t}^{\vee}=\sigma _i^*(\Omega^1_{\U_t/\F_t})^{\vee}=\sigma _i^*T_{\U_t/\F_t}.\]
It follows from Definition \ref{defpi} 
that $\sigma _i^*T_{\U_t/\F_t}=(\pi_{p_i}|_{\F_t})^*\OO_{\PP(T_{p_i}X)}(-1)$. Lemma \ref{sublinear} implies that $\sigma_i^* T_{\U_t/\F_t}=-\lambda|_{\F_t}$. We prove this lemma.

\end{proof}

\begin{lemm} \label{numertri}
The line bundle $\sigma_i^*\omega_{\U_t/\F_t}\otimes \OO_{\F_t}(-\Delta_i)$ on $\F_t$ is numerically trivial. In particular, the divisor $\Delta_{i,t}$ is an ample divisor on $\F_t$.

\end{lemm}

\begin{proof}
Let $B$ be a smooth projective curves on $\F_t$ such that $B$ and $\Delta_i$ intersect transversely at general points of $\Delta_i$ for $i=1,\ldots ,m$. We first show that $$B\cdot\sigma_i^*\omega_{\U_t/\F_t}=B\cdot \Delta_i$$.

Denote the restriction of $\U_t\rightarrow \F_t$ to $B$ by $C\rightarrow B$. Namely, we have a cartesian diagram as follows,
\[\xymatrix {C\ar[r]^g \ar@{}[rd]|-{\square} \ar[d]_{\pi_C}& \U_t \ar[d]^{\pi_t} \\
B\ar[r]_h \ar@/^2pc/[u]^{\widehat{\sigma_k}} &\F_t\ar@/_2pc/[u]_{\sigma_k} }\]
where the section $\widehat{\sigma_k}$ of $\pi_C$ is the pullback of the section $\sigma_k$ of $\pi_t$ with $g\circ \widehat{\sigma_k}=\sigma_k\circ h$. 

We claim that $C$ is smooth surface. In fact, we recall that $\Delta_t$ is a simple normal crossing divisor on $\F_t$, cf. Lemma \ref{divisorNCD}. It is clear that $\U_t$ is a family of semistable curves over $\F_t$ with discriminant locus $\Delta_t$. Note that $B$ and $\Delta_t(= \cup \Delta_{i,t})$ intersect transversely. Therefore, around a nodal point of some fiber $C_b$ of $\pi_C$ over $b\in B$, $C$ is defined by the equation $xy=s$ where $s$ is the uniformizer of $B$ at $b$ and $b\in B\cap\Delta_t$. By local calculations, the surface $C$ is smooth. 

Recall that general points of $\Delta_i$ parameterize the stable maps whose image is the union of a line containing $p_i$ and a rational normal curve, cf. Definition \ref{definedivisor}. Note that $B$ intersects with $\Delta_i$ at general points of $\Delta_i$. Therefore, the fibers of $\pi_C$ are either a smooth rational curve or the union of two smooth rational curves intersecting transversely at a point. The fiber $C_s$ of $\pi_C$ over $s\in \Delta_t \cap B$ has two components $L_s$ and $Q_s$ where $L_s$ is the unique component of the stable map parametrized by $s$ whose image in $X$ is a line and $Q_s$ is the rest component. It is clear that \[C_s=L_s+Q_s, C_s\cdot L_s=0 \text{~and~} L_s\cdot Q_s=1.\] It follows that $L_s$ is $(-1)$-curve. By Castelnuovo's contraction theorem, we have a map $\phi:C \rightarrow C'$ contracting $(-1)$-curves $L_s$ ($s\in \Delta_t\cap B$) and the map $\pi_C$ factors through $\phi$, i.e., we have that $\pi_C=q\circ \phi$ $\xymatrix{: C \ar[r]^{\phi} & C' \ar[r]^q &B}$. It is clear that $\phi \circ \widehat{\sigma_k}$ ($k=1,\ldots,m$) give at least three disjoint sections of $q:C'\rightarrow B$ and $C'/B$ is a $\PP^1$-bundle. It follows that the $\PP^1$-bundle $C'/B$ is trivial, i.e., $C'\cong B\times \PP^1$.

Now, we are able to express the dualizing sheaf $\omega _{C/B}$ in an explicit way. Namely, we have
\begin{align*}
\omega_{C/B}=K_C\otimes \pi_C^*K_B^{-1}&=\left(\phi^*K_{C'}\otimes \OO_C(\sum\limits_{i=1}^m\sum\limits_{s\in \Delta_i\cap B} L_s)\right)\otimes   \pi_C^*K_B^{-1}\\
&=\phi^*\omega_{C'/B}\otimes \OO_C(\sum\limits_{i=1}^m\sum\limits_{s\in \Delta_i\cap B} L_s)
\end{align*}
where the second equality in the first line is due to the blow-up formula for the canonical bundle. 

 Note that the composition \[\xymatrix{B\ar[r]^{\widehat{\sigma_k}} &C \ar[r]^{\phi}& C'\ar@{=}[r] &B\times \PP^1 \ar[r]^{pr_2} &\PP^1 }\] is a constant morphism and $\omega_{B\times \PP^1/B}=pr_2^*(\omega_{\PP^1})$. It follows that 
 \begin{align*}
 \widehat{\sigma_k}^* \omega_{C/B}&=\widehat{\sigma_k}^*\left(\phi^* pr_2^*(\omega_{\PP^1})\right)\otimes \widehat{\sigma_k}^* \left( \OO_C(\sum\limits_{i=1}^m\sum\limits_{s\in \Delta_i\cap B} L_s) \right)\\
 &=\OO_B\left(\sum\limits_{s\in \Delta_k\cap B} L_s \cdot \widehat{\sigma_k}(B) \right)\\
 &=\OO_B(\sum\limits_{s\in \Delta_k\cap B} s)=B\cdot \Delta_k. 
 \end{align*}
where the equalities in the second and third lines are due to the facts that $\widehat{\sigma}_k(B)$ does not intersect $L_s$ for $s \not\in \Delta_k\cap B$ and intersect $L_s$ transversely for $s \in \Delta_k\cap B$.
On the other hand, it follows from the base change property of dualizing sheaves that \[B\cdot \sigma_k^*(\omega_{\U_t/\F_t})=h^*\sigma_k^*(\omega_{\U_t/\F_t})=\widehat{\sigma_k}^*g^*\omega_{\U_t/\F_t}=\widehat{\sigma_k}^* \omega_{C/B}.\]
Therefore, we conclude that $B\cdot \left(\sigma_k^*\omega_{\U_t/\F_t}-\Delta_k\right)$ is zero.

In general, to show the proposition, it suffices to prove that the line bundle $\sigma_i^*\omega_{\U_t/\F_t}\otimes \OO_{\F_t}(-\Delta_i)$is numerically trivial on the curves given by complete intersections of a very ample divisor, cf. \cite[Page 69 Chapter 3 Remark 3.8]{Deb}. By the theorem of Bertini, we reduce to the case for which the curve is smooth and intersects with $\Delta_i$ at general points of $\Delta_i$ transversely, which we have already proved.

By Proposition \ref{propseppts} and Lemma \ref{lemmaI}, we know that $\sigma_i^*\omega_{\U_t/\F_t}$ is ample. Therefore, it follows from Kleiman's ampleness criterion that the divisor $\Delta_i$ is ample.

\end{proof}

\begin{prop}\label{cyclerelation}
Suppose that $n-m(\sum\limits_{i=1}^c d_i-c)-c\geq 2$. Then $\F_t$ is nonempty with $\OO_{\F_t}(\Delta_{i,t})=\lambda|_{\F_t}$. 
\end{prop}

\begin{proof}
Let $D$ be the space of $m$ lines in $X$ such that each line contains exactly one point among the points $p_1,\ldots,p_m$ and the lines intersect at a point. It is clear that $D$ is a non-singular complete intersection in $\PP^n$ of type
\begin{equation}\label{eq:equations}
(T_2(d_1,m),(T_2(d_2,m), \ldots, T_2(d_c,m)),
\end{equation}
see \cite[Page 83 (2)]{DS2}. 
Hence, the dimension of $D$ is $n-m(\sum\limits_{i=1}^c d_i-c)-c$ which is at least $2$.

 Let $t \in \M_{0,m}$ be a general point parametrizing a smooth rational curve with $m$ pointed points. Denote this pointed rational curve by $(R,y_1,\ldots, y_m)$ where $R$ is a rational curve and $\{y_i\}$ are the marked points on $R$. Let $u\in D$ be the point representing the union of lines $l_1\cup l_2\cup\ldots\cup\l_m$ such that $p_i\in l_i$ and the lines $\{l_i\}$ intersect at a point $Q$. We can canonically associate to $u$ the stable map $(C,f,x_1,\ldots,x_m)$ of maximal degeneration type (cf. Lemma \ref{lemmsection}) in $\F_t$ as follows:
\begin{enumerate}
   \item the domain $C$ is $R\coprod l_1\coprod l_2 \ldots \coprod \l_m$ identifying $Q(\in l_i)$ with  $y_i$, and
  \item $f(x_i)=p_i$, and
  \item the map $f$ maps $l_i$ identically to the line $l_i \subseteq X$ and collapses $R$ to the point $Q$.
\end{enumerate}

Therefore, it gives rise to a morphism
\begin{equation} \label{eqmap}
 \psi:D \rightarrow Y=\bigcap\limits_{i=1}^m\Delta_{i,t}\subseteq \F_t.
\end{equation}
This morphism is bijective. Therefore, it is an isomorphism by the smoothness and irreducibility of $D$ and Lemma \ref{divisorNCD}. Note that 
\begin{align*}
dim(\F_t)&=dim(\F)-dim(\M_{0,m})= (c+1-\sum\limits_{i=1}^c d_i)m+n-c.
\end{align*}
by Lemma \ref{lemm0}
i.e., $dim(Y)=dim(D)=dim(\F_t)-m$. 

Since the divisors $\Delta_i$ are ample, the intersections $\bigcap\limits_{i=1}^s\Delta_i$ are smooth and connected for $s=1,\ldots,m$.
Since $D$ is a smooth complete intersection of dimension at least $2$ in a projective space, we have that $H^2(Y,\mathbb{Z})$ is torsion-free and $H^1(Y,\OO_Y)=0$. By the Lefschetz hyperplane theorem and induction on s, we have that $H^1(\F_t,\OO_{\F_t})=0$, $H^2(\F_t,\mathbb{Z})$ is torsion-free and the map of the first Chern class $c_1:\Pic(\F_t)\rightarrow H^2(\F_t,\mathbb{Z})$ is injective. It follows that $\OO_{\F_t}(\Delta_{i,t})=\lambda|_{\F_t}$ by Lemma \ref{numertri} and the Poincar\'e duality theorem. The proposition follows from Lemma \ref{lemmaI}.

\end{proof}

From the proof of Proposition \ref{cyclerelation}, we conclude the following corollary. 
\begin{cor}\label{corintersectype}
Suppose that $Y$ is the locus in $\F_t$ parametrizing the stable maps of maximal degeneration type, i.e., the image of the stable map is the union of $m$ lines which intersect at a point. Assume that $n+m(c-\sum\limits_{i=1}^c d_i)-c\geq 2$.
\begin{itemize}
\item The intersection $\bigcap\limits_{i=1}^{s}\Delta_{i,t}$ is smooth and connected for $s=1,\ldots m$. The locus $Y$ is $\bigcap\limits_{i=1}^{m}\Delta_{i,t}$ of dimension $n+m(c-\sum\limits_{i=1}^c d_i)-c$.  
\item   Via the map  $\Phi|_{Y}:Y \hookrightarrow \PP^{n-m}$, the variety $Y$ is a complete intersection in $\PP^{n-m}$ defined by the disjoint union of the homogeneous polynomials of type (see \ref{hompoly})\[(T_1(d_1,m), T_2(d_2,m),\ldots, T_2(d_c,m)).\]
\end{itemize}
\end{cor}

\begin{proof}
The first assertion follows from the proof of Proposition \ref{cyclerelation}.
With the same notations as in the proof of Proposition \ref{cyclerelation}, we recall that $D$ is a smooth complete intersection in $\PP^n$ of type as (\ref{eq:equations}). To show the second assertion, we consider the following map
\[\xymatrix{
Y=D \ar@{^(->}[r]^i & \PP^n\ar@{.>}[r]^{pr} &\PP^{n-m}
}
\]
where $i$ is the natural inclusion and $pr$ is the projection from \[\Span(p_1,\ldots,p_m)=\PP^{m-1}\] to a projective subpace $\PP^{n-m}$ of $\PP^n$ disjoint from $\Span(p_1,\ldots,p_m)$.  We can choose this $\PP^{n-m}$ to be the intersection of $m$ hyperplanes defined by linear forms in $T_2(d_1,m)$. To be precise, these $m$ hyperplanes are projective hyperplanes tangent to the hypersurface $F_1=0$ at $p_1,\ldots,p_m$ respectively, where $F_1$ is a homogeneous polynomial of degree $d_1$ for defining $X$ as a complete intersection $X=V(F_1,\ldots, F_c)$.

We claim that the composition $pr\circ i$ is the morphism $\Phi|_Y$. Therefore, the second assertion follows. 

In fact, let $(C,f,x_1,\dots, x_m)$ be a stable map parametrized by $Y$. Suppose that the image $f(C)$ of $C$ is the union of lines $l_1\cup l_2\cup \ldots \cup l_m$ such that the lines $l_i$ intersect at $Q$. The projective space $\Span(f(C))$ is given by $\Span(p_1, \ldots, p_m, Q)$. Note that $D$ is the intersection $\bigcap\limits^m_{i=1}L_i$ of $L_i$ where $L_i$ is the union of lines on X passing through $p_i$. Therefore, $Q$ is on $D$ and $\psi(Q)=[(C,f,x_1,\dots, x_m)]$, see (\ref{eqmap}). Therefore, we have that
\begin{align*}
pr\circ i ([(C,f,x_1,\dots, x_m)] )&=pr (Q)=\Span(p_1,\ldots,p_m,Q)\cap \PP^{n-m}\\
&=\Span(f(C))\cap \PP^{n-m}.
\end{align*} 
On the other hand, we can identify $\PP^n/\Span(p_1,\ldots,p_m)$ with $\PP^{n-m}$ by sending a point \[P\in \PP^n/\Span(p_1,\ldots,p_m)\] parametrizing a projective space $P$ of dimension $m$ and containing $\Span(p_1,\ldots,p_m)$ to the point$P\cap \PP^{n-m}$. Under this identification, we conclude that $$pr\circ i([(C,f,x_1,\dots, x_m)] )=\Phi|_Y([(C,f,x_1,\dots, x_m)] ).$$ 



\end{proof}

 \section{Embedding Maps} \label{sectionembd}

In the section, we show that the map induced by $|\lambda|_{\F_t}|$ is an embedding, see Proposition \ref{propembed}. Using Lemma \ref{lemmaembed} and the deformation theory of stable maps, we show that the complete linear system $|\lambda|_{\F_t}|$ is unramified, see Proposition \ref{propembed}.

The following lemma helps us to show that $|\lambda|_{\F_t}|$ separates tangents.

\begin{lemm} \label{lemmaembed}
Let $Y$ be a smooth projective manifold with a line bundle $L$ and irreducible divisors $D_i$ for $i=1,\ldots k$. If the following conditions are satisfied
\begin{itemize}
\item the complete linear system of $L$ defines a morphism \[h:Y\rightarrow \PP(H^0(Y,L)^{\vee})=\PP^N\] with $D_i\in |L|$ and $D_i$ are reduced,
\item the differential $Dh:T_Y\rightarrow h^*T_{\PP^N}$ of $h$ is injective over $Y-(\bigcup\limits_{i=1}^m D_i$),
\item the restriction $Dh|_{T_{D_i,q}}: T_{D_i,q}\rightarrow T_{\PP^N,h(q)}$ of $Dh$ to a general point of $D_i$ for each $i$ is injective,

\end{itemize}
then the morphism $h$ is unramified (i.e., the differential map $dh$ is injective).
\end{lemm}
\begin{proof}
Let $g$ be a map between two complex manifolds $g:Y\rightarrow Z$. Recall that the ramfied locus of $g$ on $Y$ is of codimension one or $Y$. Therefore, by the second condition, it suffices to show that the differential $Dh|_{T_{Y,s}}:T_{Y,s}\rightarrow T_{\PP^N,s}$ is injective at a general point $s$ of $D_i$ for each $i$. 

Suppose that there is a nonzero vector $v\in T_{Y,s}$ with $Dh(v)=0$ where $s$ is a general point of $D_i$. We claim that it is impossible. 

In fact, note that $D_i$ is generically smooth. The third condition implies that $v$ and $T_{D_i,s}$ intersect transversally. We choose a small disk $\mathbb{D}(\subseteq \C)$ in $Y$ such that its tangent vector at $0\in \mathbb{D}$ is $v$, i.e., we have an embedding $i:\mathbb{D} \hookrightarrow Y$ with $Di(w)=v$ for a nonzero tangent vector $w$ to $\mathbb{D}$ at $0\in \mathbb{D}$ ($i(0)=s$). Therefore, the disk $\mathbb{D}$ and $D_i$ intersect transversally at $0\in \mathbb{D}$. In other words, the function $i^*(G)$ has only simple zero at $0\in \mathbb{D}$ where $G$ is the local function around $s\in Y$ defining $D$.

 On the other hand, let $H$ be the hyperplane in $\PP^N$ with $h^{-1}(H)=\Delta$. Suppose that $H$ is locally defined by $L=0$ around $h(s)$. The equation $h^*(L)=0$ defines $\Delta$ locally. We may assume $h^*(L)=G$. Since we have $D(h\circ i)(w)=Dh(v)=0$, the order of the zero of the function $i^*(h^*(L))=i^*(G)$ is at least two at $0\in \mathbb{D}$. It is a contradiction.

\end{proof}

We should discuss the deformation theory of stable maps. The general reference for the deformation problems of maps is \cite{I}. Specializing in the case of stable maps, we refer to \cite[Page 61]{GHS}, \cite{BF} and \cite{B}.
Let $f:C\rightarrow X$ be an unmarked stable map. By \cite{BF} and \cite{B}, the space of first-order deformations and the obstruction group are given by the hypercohomology groups
\begin{align*}
\Def(f)=\mathbb{H}^1(C,\mathbb{R}\Hom_{\OO_C}(\Omega^{.}_f,\OO_C)),\\
\Obs(f)=\mathbb{H}^2(C,\mathbb{R}\Hom_{\OO_C}(\Omega^{.}_f,\OO_C)).
\end{align*}
where $\Omega^{.}_f$ is the complex
$\xymatrix{f^*\Omega_X^1 \ar[r]^{df}&\Omega_C^1 } .$

If $f$ is an embedding , then $\mathbb{R}\Hom_{\OO_C}(\Omega^{.}_f,\OO_C)\cong N_f[-1]$where $N_f$ is the normal bundle of $f$, and the space of the first order deformations of a stable map $(C,f,x_1,\ldots,x_m)$ is $H^0( C, N_f)$.
Furthermore, the space of the first order deformations of $(C,f,x_1,\ldots,x_m)$ fixing the points $\{x_i\}_{i=1}^m$ is given by 
\begin{equation}\label{deformation}
H^0\left(C,N_f(-\sum\limits_{i=1}^m x_i)\right).
\end{equation}

\begin{prop} \label{propembed}
With the same notations as Proposition \ref{propseppts}, the complete linear system of $\lambda|_{\F_t}$ separates tangent vectors of $\F_t$, i.e., the differential of the map induced by the complete linear system $|\lambda|_{\F_t}|$ is injective. In particular, together with Proposition \ref{propseppts}, the morphism $|\lambda|_{\F_t}|$ is a closed embedding.
\end{prop}

\begin{proof}

We apply Lemma \ref{lemmaembed} to $Y=\F_t$, $D_i=\Delta_i$ and $L=\lambda|_{\F_t}$ where $\{D_i\}$ are irreducible components of $\Delta_t$. It suffices to check the conditions of Lemma \ref{lemmaembed}. The first condition is verified by Lemma \ref{remark2}, Corollary \ref{corintersectype} and Proposition\ref{cyclerelation}. 

For the second condition, we claim that the morphism $h$ induced by $|\lambda|_{\F_t}|$ is unramified on $U=\F_t-(\bigcup\limits_{i=0}^m \Delta_i)$. In fact, by the fact that $\pi_{p_i}^*(\OO_{\PP(T_{p_i}X)}(1))=\lambda|_{\F_t}$, cf. Lemma \ref{sublinear}, it suffices to show that the morphism\[
\xymatrix{
\F_t \ar[rrr]^(0.35){\pi_{p_1}\times\pi_{p_2}\times \ldots,\times \pi_{p_m}}  & & & \PP(T_{p_1}X)\times\ldots\times\PP(T_{p_m}X).
}
\]
is unramified on $U$. The kernel of the differential $D(\pi_{p_1}\times\pi_{p_2}\times \ldots,\times \pi_{p_m})$ at the point $(C,f,x_1,\ldots,x_m)\in \F_t$ is a subspace of the vector space\[\Def_{(C,f,x_1,\ldots,x_m) }\] parametrizing the first order deformations of a stable map $(C,f,x_1,\ldots,x_m)$ fixing both the points $\{x_i\}_{i=1}^m$ and the tangent directions at $\{x_i\}$. Similarly, as (\ref{deformation}), we have that 
\begin{equation}\label{eqdeform}
\Def_{(C,f,x_1,\ldots,x_m) }=H^0(C, N_{C/X}(\sum\limits_{i=1}^m -2x_i))
\end{equation}
if $f$ is an embedding, where $N_{C/X}$ is the normal bundle $N_f$ of $f:C\hookrightarrow X$. Note that $H^0(C, N_{C/X}(\sum\limits_{i=1}^m -2x_i))=0$ if $m\geq 3$ and $C$ is smooth by Corollary \ref{corovanish}. 

Now, we verify the third condition of Lemma \ref{lemmaembed}. It is obvious that the morphisms $h|_{\Delta_i}:\Delta_i \rightarrow \PP^N$ ($i=1,\ldots,m$) are generically unramified since the map $h=|\lambda|_{\F_t}|$ is injective by Proposition \ref{propseppts}. The third condition holds.

In summary, we prove the proposition.
\end{proof}

\section{The Main Theorem} \label{sectionmainthm}

In this section, we prove Theorem \ref{mainthm}. An important ingredient is a criterion (Proposition \ref{mainprop}) for characterizing when a smooth projective variety is a complete intersection in a projective space. Note that $\F_t$ contains a subvariety $Y$ which is a complete intersection $\bigcap\limits_{i=1}^m\Delta$ in $\F_t$. Note that $\Delta_i$ are very ample divisors and $Y$ is a complete intersection in a projective space, see Corollary \ref{corintersectype}. Theorem \ref{mainthm} follows from Corollary \ref{corintersectype} and Proposition \ref{mainprop}, see the proof at the end of this section for more details.

\begin{prop}\cite[Proposition 7.3]{PAN1} \label{mainprop}
 Suppose that we have smooth projective varieties $\Delta$ and $\F$ in $\PP^N$. Assume that
\begin{enumerate}
  \item the variety $\Delta$ is a smooth divisor of $\F$ with dimension at least $1$,
  \item the divisor $\Delta$ is a complete intersection in $\mathbb{P}^N$ of type $(d_1,\ldots,d_c)$ with $d_i\geq1$,
   \item the divisor $\Delta$ is the schematic intersection of $\F$ and a hypersurface of degree $d_1$ in $\PP^N$.
\end{enumerate}
Then, the smooth variety $\F$ is a complete intersection of type $(d_2,\ldots, d_c)$ in $\mathbb{P}^N$.
\end{prop}

\begin{prop} \label{propembedding}
Use the same notations as in Corollary \ref{corintersectype}. If $n+m(c-\sum\limits_{i=1}^c d_i)-c\geq 2$, then the map
\[|\lambda|_{\F_t}|:\F_t \rightarrow \PP^N\]
is an embedding and $\dim H^0(Y,\OO_Y(1))=n-mc$ where $N$ is $\dim |\lambda|_{\F_t}|$. 

 Moreover, the varieties $\F_t$ and $\bigcap\limits_{i=1}^k\Delta_{i,t}$ ($1\leq k\leq m$) are smooth complete intersections via this embedding.
\end{prop}

\begin{proof}
By Corollary \ref{corintersectype}, we have an embedding $\Phi|_Y:Y\hookrightarrow \PP^{n-m}.$
The image of this embedding is in a projective space $\PP^{n-mc}$. Moreover, it is a complete intersection in $\PP^{n-mc}$ of type\\
\begin{equation} \label{eq:equations1}
(T_1(d_1,m), T_1(d_2,m),\ldots, T_1(d_c,m)).
\end{equation}
\\
Therefore, it follows from Lemma \ref{lemcha3} that dim $H^0(Y,\OO_Y(1))=n-mc$. We consider the following embedding
\[|\lambda|_{\F_t}|_{Y}:Y \rightarrow \PP^N.\]
The space $Y$ is a complete intersection in $\PP^N$ defined by the union of some hyperplanes and some polynomials of type (\ref{eq:equations1}). Note that $Y=\bigcap\limits_{i=1}^{m} \Delta_{i,t}$ and $$n+m(c-\sum\limits_{i=1}^{c} d_i)-c\geq 2.$$ It follows that $\dim Y\geq 1$. The intersection  $\bigcap\limits_{i=1}^{k}\Delta_{i,t}$ is a smooth projective variety for $1\leq k \leq m$. Hence, by Proposition \ref{cyclerelation} and Proposition \ref{propembed}, we can apply Proposition \ref{mainprop} to $\bigcap\limits_{i=1}^{k}\Delta_{i,t}$ inductively, from $k=m$ to $k=0$, where $\bigcap\limits_{i=1}^{k}\Delta_{i,t}=\F_t$ if $k=0$. \\

It follows that $\bigcap\limits_{i=1}^{k}\Delta_{i,t}$ and $\F_t$ are complete intersections in $\PP^N$. 
They are defined by the union of some polynomials of type (\ref{eq:equations1}) and some linear forms.

\end{proof}

\begin{prop}
With the same hypothesis as in Proposition \ref{propembedding}, we have \[N=\dim |\lambda|_{\F_t}|=n-m(c-1).\]
\end{prop}

\begin{proof}
Denote $\bigcap\limits_{i=1}^{k}\Delta_{i,t}$ by $Y_k$ , and let $Y_0$ be $\F_t$. Note that $Y_k$ is a smooth projective variety with the embedding $|\lambda|_{\F_t}|:Y_0\rightarrow \PP^N$ by Corollary \ref{corintersectype} and Proposition \ref{propembedding}, and $\OO_{\F_t}(\Delta_{i,t})=\lambda|_{\F_t}=\OO_{Y_0}(1)$ by Proposition \ref{cyclerelation}. We have the following short exact sequence
\[0\rightarrow\OO_{Y_k}(-\lambda|_{Y_k})\rightarrow \OO_{Y_k}\rightarrow j_*\OO_{Y_{k+1}}\rightarrow 0,\]
where $j$ is the inclusion $Y_{k+1}\subseteq Y_k$.
Note that $Y_k$ is a complete intersection of dimension 
$>1$ for $k\leq m-1$. It follows that $H^1(Y_k,\OO_{Y_k})=0$. We tensor the short exact sequence above with $\OO_{Y_k}(1)$ and take $H^i(\_)$. It gives rise to a long sequence
\[0\rightarrow H^0(Y_k,\OO_{Y_k})\rightarrow H^0(Y_k, \OO_{Y_k}(1))\rightarrow H^0(Y_{k+1},\OO_{Y_{k+1}}(1))\rightarrow H^1(Y_k,\OO_{Y_k})=0.\]
Therefore, we conclude that \[\dim |\lambda|_{\F_t}|=\dim H^0(\F_t,\OO_{\F_t}(1)))=\dim H^0(Y_m,\OO_{Y_m}(1))+m=n-mc+m\]
where the last equality follows from Proposition \ref{propembedding}.

\end{proof}

We are able to show our main theorem \ref{mainthm} now.

\begin{proof} \label{thmproof}
With the notations as above, the general fiber $Y_0=\F_t\subseteq \PP^N$ is a complete intersection in $\PP^N$ defined by homogeneous polynomials of type (\ref{eq:equations1}) with $s$ linear forms (see the end of the proof of Proposition \ref{propembedding}). Note that
$$\dim\F_t=\dim\F-\dim\M_{0,m}=(c+1-\sum\limits_{i=1}^{c} d_i)m+n-c$$ and $N=n-m(c-1)$. Note that $Y_m=\bigcap\limits^m_{i=1} \Delta_i$ is a smooth complete intersection in $\PP^n$ of type (\ref{eq:equations}). The type (\ref{eq:equations}) consists of $m(\sum\limits_{i=1}^{c} d_i -2c)$ polynomials. It follows from the dimension count that
\[ s=N-\dim\F_t-m(\sum\limits_{i=1}^{c} d_i -2c)-c=0.\]
Therefore, we show the theorem.

\end{proof}

\section{Applications} \label{sectionapp}
We have three interesting applications of our main Theorem. The first one is related to the rational connectedness of the moduli spaces of rational curves on varieties. The rational connectedness of the moduli spaces of rational curves on varieties is arising from some arithmetic problems, such as weak approximation and the existence of rational points on varieties over function fields.

The second application towards to enumerative geometry. We give a proof of a classical formula to count the number of twisted cubics on a complete intersection (see the paper \cite{BV} and \cite{FP}). Moreover, we provide a new formula for counting the number of two crossing conics on a complete intersection.

The third one is to show the Picard group of $\F$ is finitely generated.\\

\textbf{The Rational Connectedness of Moduli Spaces}
\medskip

 \begin{prop} \label{propRC}
 With the same hypothesis as in Theorem \ref{mainthm}, if \[m\left(\sum\limits_{i=1}^c\frac{d_i(d_i-1)}{2}-1\right)+\sum\limits_{i=1}^c d_i \leq n,\]
 then $\F$ is rationally connected.
 \end{prop}

 \begin{proof}
With the same notations as in Theorem \ref{mainthm}, the canonical bundle of the fiber $\F_t$ is given by
\[ K_{\F_t}=\OO_{\PP_{N}}\left(-N-1+m\left(\sum\limits_{i=1}^c(\frac{d_i(d_i-1)}{2}-1)\right)+\sum\limits_{i=1}^c d_i\right)\]
where $t\in \M_{0,m}$ is a general point as in Theorem \ref{mainthm}. The inequality in the hypothesis of the proposition is equivalent to say $K_{\F_t}$ is anti-ample. In particular, the fiber $\F_t$ is a smooth projective Fano variety, hence, it is rationally connected, see \cite[Chapter V]{K}. We claim that the fiber $\F$ is rationally connected.

In fact, we suppose that $p$ and $q$ are general points of $\F$ such that $F(p)$ and $F(q)$ are general points in $\M_{0,m}$. Since $\M_{0,m}$ is a smooth projective and rational variety, there is a rational curve $D$ in $\M_{0,m}$ connecting $F(p)$ and $F(q)$ such that the fiber of $F$ over a general point $D$ is rationally connected. Therefore, by Corollary $1.3$ in \cite{GHS}, the general fiber $\F$ is rationally connected.
\end{proof}

\begin{remark}
By \cite[Lemma 6.5]{DS}, it is not hard to prove that the canonical bundle $K_\F$ is trivial on some rational curves sitting inside the maximal degeneration locus if $m\geq 4$. Therefore, $\F$ is not Fano for $m\geq 4$.
\end{remark}

\textbf{Enumerative Geometry}
\medskip

Suppose that $X'$ is a complete intersection which is cut out from a complete intersection $X$ by $n-s$ general hyperplanes. Assume $s\geq m$. We denote the general fiber of evaluation map corresponding to $X'$ by $\F'$. It is clear that $\F'$ is cut out from the general fiber $\F$ (corresponding to $X$) by $n-s$ very ample divisors in $|\lambda|$.
\begin{equation} \label{eq:generalsections}
\xymatrix{\F' \ar@/^2pc/[rr]|-{F'} \ar[d]^{\Phi'} \ar@{^{`}->}[r] & \F\ar[d]^{\Phi} \ar[r]^{F} &\M_{0,m} \\
\PP^{s-m} \ar@{^{`}->}[r]& \PP^{n-m}
}
\end{equation}

If $\F'$ consists of discrete points, i.e., $\dim \F'$=0, then the number of these points is the number of the rational curves of degree $m$ passing through $m$ general points on $X'$. See the following proposition for the precise statement.

\begin{prop}\cite[Collary, page 9]{BV}\label{countcubic}
Let $X$ be a smooth complete intersection of degree $(d_1,\ldots,d_r)$ in $\mathbb{P}^{n+r}$, with $n=3\sum\limits_{i=1}^r(d_i-1)-3$. Then the number of twisted cubics in $X$ passing through 3 general points $(p,q,r)$ is $\frac{1}{d^2}\prod\limits_{i=1}^r (d_i!)^3$ where d is the degree of $X$.

\end{prop}

\begin{proof}

We can assume that the variety $X$ is cut out by hyperplanes from a smooth complete intersection $Y$ of type $(d_1,\ldots,d_r)$ and $Y\subseteq \PP^e$ has sufficiently large dimension. The degree of the general fiber $\F\subseteq \mathbb{P}^{e}$ corresponding to $Y$ has an enumerative geometrical interpretation.

In fact, a point in the $\mathbb{P}^{e-3}=\PP^e/\Span(p,q,r)$ parametrizes a 3-plane $\PP^3$ in $\PP^e$ containing the 2-plane $\PP^2=\Span(p,q,r)$. More generally, a sub-projective space $\mathbb{P}^k(\subseteq \mathbb{P}^{e-3})$ corresponds to a $k+3$-plane $\mathbb{P}^{k+3}$ in $\PP^e$ (containing $\Span(p,q,r)$). So if we take a $l$-plane $\mathbb{P}^l\subseteq \PP^e$ 
such that $l+\text{dim}(\F)=e$, then $\#(\PP^l\cap \F)$ is the number of twisted cubics in $X=\mathbb{P}^{l+3}\cap Y$ passing through $p, q$ and $r$. In other words, the degree of $\F$ is equal to
\begin{center}
{
 \# \{twisted cubics in $X$ passing through $p, q$ and $r$\}.
}
\end{center}
By Theorem \ref{mainthm}, the degree of $\F$ is
$\frac{1}{d^2}\prod\limits_{i=1}^r (d_i!)^3$.
We show the proposition.
\end{proof}

\begin{remark} \label{rmklinkconic}
One can get a similar formula for linking conics on $X$ passing through $4$ general points. More precisely, the number of these linking conics is $\frac{1}{d^3}\prod\limits_{i=1}^r (d_i!)^4$ where $d$ is the degree of $X$ where $X$ is a smooth complete intersection of degree $(d_1,\ldots,d_r)$ in $\mathbb{P}^{n+r}$ with $n=4\sum\limits_{i=1}^r(d_i-1)-4$.
\end{remark}
\textbf{The Picard Group of Moduli Spaces}
\medskip\\
The following lemma is standard and easy to prove. We omit the proof.
\begin{lemm}\label{lemmapicfinte}
Consider a morphism $h:A\rightarrow B$ between two smooth varieties $A$ and $B$ over $\mathbb{C}$.
Suppose that the morphism $h$ is proper and dominant. Let $K$ be the function field of $B$. If the following two conditions are satisfied
\begin{enumerate}
  \item the generic fiber $A_{K}$ of $h$ is geometrical connected,
  \item the Picard groups $\Pic(A_{K})$ and $\Pic(B)$ finitely generated,
\end{enumerate}
then the Picard group $\Pic(A)$ is finitely generated.

\end{lemm}

\begin{lemm} \label{lemmci}
Suppose that $U$ is a variety with a flat and projective morphism $g$ as follows.
\[\xymatrix{Y  \ar[d]_{g}\ar@{}[r]|-{\subseteq}&\PP^n_{U}\ar[dl]\\
U }\]
We assume that the fiber $Y_s$ over any $s\in U(\mathbb{C})$ is a complete intersection of type $(d_1,\ldots,d_c)$ in $\PP^n_{s}$. If $K$ is the function field of $U$, then the generic fiber $Y_{K}$ of $g$ is a complete intersection of type $(d_1,\ldots,d_c)$ in $\PP^n_{K}$.

\end{lemm}

\begin{proof}
We first notice that if a projective variety $T$ in $\PP^n$ is a complete intersection of type $(d_1,\dots, d_g)$ defined by homogenous polynomials $(F_1,F_2,\ldots,F_g)$, then , by Hilbert's theory (see \cite[Chapter I section 7]{H}), the polynomials $(F_1,F_2,\ldots,F_g)$ are minimal generators of the ideal sheaf $I_T$ and the dimension of $H^0(\PP^n,I_T(m))$ is only dependent on the type $(d_1,\dots, d_g)$.

Therefore, the dimension function  \[s\mapsto \dim H^0(\PP^n_s,I_{Y_s}(m))\]
is a constant function ($s\in U(\mathbb{C})$) for any $m$. It follows from \cite[Page 48, Corollary 2]{AVB} that
\[\dim H^0(\PP^n_K,I_{Y_K}(m))=\dim H^0(\PP^n_s,I_{Y_s}(m)).\] This equality implies that we can choose homogenous polynomials \[(F_1,F_2,\ldots,F_c)\] defined over $K$ such that, on the domain $V(\subseteq U)$ of the coefficients of $F_i$, the fiber $Y_s$ of $g$ over a point $s\in V(\mathbb{C})$ is a complete intersection in $\PP_s^n$ defined by the equations as follows:\[(F_1,F_2,\ldots,F_c)|_{\PP^n_s}.\]

It follows that the projective variety $Y_K$ is a complete intersection in $\PP^n_K$ defined by $(F_1,F_2,\ldots,F_c)$. We prove the lemma.
\end{proof}

\begin{lemm}\label{lemmpicgen}
Consider a projective morphism $h$ from a scheme $A$ to $\Spec(K)$ where $K$ is a field. If the following two conditions hold,
\begin{enumerate}
  \item the morphism $h$ has a section $\sigma$
  \item $\Pic(A_{\overline{K}}$)=$\mathbb{Z}[L|_{A_{\overline{K}}}]$=$\mathbb{Z}$ where $L$ is a line bundle on $A$ and $\overline{K}$ is the algebraic closure of $K$,
\end{enumerate}
then $\Pic(A)=\mathbb{Z}=\mathbb{Z}[L]$.
\end{lemm}

\begin{proof}
The hypothesis (1) of the lemma ensures that the Picard functor is representable by the Picard scheme $\Pic_{A/K}$.
Therefore, the Picard group $\Pic(A)$ (resp. $\Pic(A_{\overline{K}})$) is just the $K$(resp. $\overline{K}$)-points of $\Pic_{A/K}$. In other words, we have that
\[\Pic(A)=\Pic_{A/K}(K)~\text{and}~\Pic (A_{\overline{K}})=\Pic_{A/K}(\overline{K}).\]
It follows that $\Pic (A)=\Pic(A_{\overline{K}})^{\Gal(\overline{K}/K)}=\mathbb{Z}[L|_{A_{\overline{K}}}]^{\Gal(\overline{K}/K)}=\mathbb{Z}[L]$
where the last equality follows from the fact that the line bundle $L|_{A_{\overline{K}}}$ is $\Gal(\overline{K}/K)$ invariant. We prove the lemma.
\end{proof}

\begin{prop}\label{picfinite}
 With the same hypothesis as in Theorem \ref{mainthm}, the Picard group $\Pic(\F)$ is finitely generated.
\end{prop}
\begin{proof}
Let $K$ be the function field of $\M_{0,m}$. It follows from Theorem \ref{mainthm} and Lemma \ref{lemmci} that the generic fiber $\F_K$ of the forgetful map $F$ is a complete intersection. The hypothesis of Theorem \ref{mainthm} implies that $\dim \F_K$ is at least $3$.

It follows from the equality $\Pic(\F_{\overline{K}})=\mathbb{Z}=\mathbb{Z}[\OO_{\PP^N_{\overline{K}}}(1)]=\mathbb{Z}[\lambda_{\overline{K}}]$ and Lemma \ref{lemmpicgen} that $\Pic(\F_K)$ is $\mathbb{Z}[\lambda_K]=\mathbb{Z}$. 

If we take $\F=A$ and $\M_{0,m}=B$ in Lemma \ref{lemmapicfinte}, then we conclude
\begin{enumerate}
  \item $\Pic(\F)$ is finitely generated, 
  \item and the restriction morphism $r^*:\Pic(\F)\rightarrow \Pic(\F_K)$ is surjective.
\end{enumerate}

We prove the proposition.
\end{proof}

\bibliographystyle{alpha}
\bibliography{mybib}
{}

\end{document}